\documentclass[11pt,a4paper]{article}
\usepackage{amsthm, amsmath, amssymb, amsfonts, url, booktabs, tikz, setspace, fancyhdr, bm, bbm}
\definecolor{brightcerulean}{rgb}{0.11, 0.67, 0.84}
\definecolor{bondiblue}{rgb}{0.0, 0.58, 0.71}
\definecolor{darkspringgreen}{rgb}{0.09, 0.45, 0.27}
\definecolor{darkred}{rgb}{0.55, 0.0, 0.0}
\colorlet{mdtRed}{red!50!black}
\usepackage[colorlinks, citecolor=darkspringgreen, linkcolor=darkred, urlcolor=darkred]{hyperref}
\usepackage{geometry}
\newtheorem{theorem}{Theorem}[section]
\newtheorem{lemma}[theorem]{Lemma}

\newtheorem{proposition}[theorem]{Proposition}
\newtheorem{corollary}[theorem]{Corollary}
\theoremstyle{remark}
\newtheorem{remark}[theorem]{Remark}

\newcommand{\E}{\mathbb{E}}
\newcommand{\Var}{\mathrm{Var}}
\newcommand{\Ent}{\mathrm{Ent}}
\newcommand{\tr}{\mathrm{tr}}
\newcommand{\supp}{\mathrm{supp}}
\newcommand{\R}{\mathbb{R}}

\newcommand{\1}{\mathbf 1}

\newcommand\tup[1]{\left\langle #1 \right\rangle}

\usepackage{mathdef}

\begin{document}

\title{Quantum Talagrand-type Inequalities via Variance Decay}
\author{Fan Chang\thanks{School of Statistics and Data Science, Nankai University, Tianjin, China; and Extremal Combinatorics and Probability Group, Institute for Basic Science, Daejeon, South Korea. \textit{E-mail: \href{mailto:1120230060@mail.nankai.edu.cn}{1120230060@mail.nankai.edu.cn}}},\quad Peijie Li\thanks{Department of Mathematics, University of Hong Kong, Hong Kong, China. \textit{E-mail: \href{mailto:lipeijie98@connect.hku.hk}{lipeijie98@connect.hku.hk}}}}
\date{}

\maketitle
\begin{abstract}
We establish dimension-free quantum Talagrand-type inequalities with explicit constants on the quantum Boolean cube, via a unified variance-decay perspective. For individual observables, short-time variance decay along the depolarizing semigroup, with rates estimated through hypercontractivity, naturally yields Talagrand-type bounds. Within this framework, we derive Talagrand-type energy--variance and high-order influence--variance inequalities. From the former, we obtain quantum analogues of Talagrand's isoperimetric inequality, the Eldan--Gross inequality, and the Cordero-Erausquin--Eskenazis inequality; from the latter, we derive high-order quantum Talagrand--KKL-type and partial isoperimetric bounds. Altogether, our work provides a variance-decay framework of broad applicability, unifying first-order and high-order Talagrand-type phenomena.

\vspace{10pt}


\vspace{10pt}

\end{abstract}

\section{Introduction}

Talagrand-type inequalities form a central lineage in the analysis of Boolean functions. Initiated by Talagrand's seminal works~\cite{Talagrand1993Iso,Talagrand1994OnRA,Talagrand1997OnBA}, these inequalities provide sharp dimension-free bounds relating variance to functionals of energy and influence, often with logarithmic amplification. They strengthen classical results such as the Poincar\'{e} inequality and the Kahn--Kalai--Linial (KKL) theorem~\cite{KKL1988}, and have profound implications for isoperimetry, sharp threshold phenomena, and concentration of measure in high-dimensional product spaces.

Broadly speaking, Talagrand-type inequalities may be divided into two families. \emph{Energy--variance inequalities} (e.g., Talagrand's isoperimetric inequality, the Eldan--Gross inequality~\cite{Eldan2022Concentration}, and the Cordero-Erausquin--Eskenazis inequality~\cite{CE2023}) relate energies (moments of gradients) to variance. \emph{Influence--variance inequalities} (e.g., the Talagrand--KKL inequality~\cite{Talagrand1994OnRA} and its $L^p$~\cite{CL2012hypercontr} and high-order~\cite{Tanguy2020,P2025} variants) quantify how local influences (moments of partial derivatives) govern global variance. Together, these two strands form a unifying framework for concentration and boundary phenomena.

The extension of these principles to noncommutative settings has recently attracted considerable attention. Montanaro and Osborne~\cite{MO2010A} initiated the systematic study of analysis on the quantum Boolean cube, and subsequent works have established quantum analogues of classical Talagrand-type inequalities. On the influence--variance side, Rouz\'{e}, Wirth and Zhang~\cite{Rouze2024quantum} proved a sub-$L^2$ quantum counterpart of Talagrand--KKL-type inequality via the semigroup approach from~\cite{CL2012hypercontr}, which yield quantum KKL and Friedgut type bounds; Blecher, Gao and Xu~\cite{Blecher2024geometricinfluencesquantumboolean} sharpened these inequalities via a quantum random-restriction method and further developed quantum analogues of the high-order influence result of Przyby\l{}owski~\cite{P2025}. On the energy--variance side, Jiao, Lin, Luo and Zhou~\cite{jiao2024quantum,JLZ2025} systematically studied quantum generalizations of classical results including Talagrand's isoperimetric inequality, the Eldan--Gross inequality, and the Cordero-Erausquin--Eskenazis inequality. These advances demonstrate that Talagrand-type inequalities retain their central role in the quantum setting, motivating the search for dimension-free formulations and high-order extensions.

In this work, we establish Talagrand-type inequalities on the quantum Boolean cube. To state our results, we first recall some basic notations (detailed definitions are given in Section~\ref{sec:preliminaries}). For observables in $M_2(\mathbb{C})^{\otimes n}$, the normalized trace $\tau$ and normalized Schatten-$p$ norms $\left\|\,\cdot\,\right\|_p$ replace the expectation and $L^p$ norms. The depolarizing semigroup $(P_t)_{t\ge0}$ plays the role of the heat flow, with generator $\mathcal L=\sum_{j=1}^n d_j$ where $d_j$ is the discrete derivative in the $j$-th coordinate. Moreover, let $|\nabla \cdot |^2=\sum_{j=1}^n |d_j \cdot|^2$ be the discrete gradient, and for a nonempty subset $J\subseteq[n]$, let $d_J=\prod_{j\in J}d_j$ be the high-order derivative over the subset of coordinates. Our main results (with explicit constants given in Section~\ref{sec:proofs-main}) include:
\begin{itemize}
    \item An \emph{energy--variance inequality} (Theorem \ref{thm:energy-variance}) that relates the $p$-energy $\Ecal_p[A]=\left\||\nabla A|\right\|_p^p$ to the variance $\Var(A)=\tau(|A-\tau(A)|^2)$: for all $A\in M_2(\mathbb{C})^{\otimes n}$ and $1\leq q\leq p\leq 2$,
    \begin{equation*}
    \|A\|_\infty^{2-p}\,\Ecal_p[A] \;\gtrsim\; \Var(A)\,\max\bigl\{1,\;\Rscr(A,q)^{p/2}\bigr\},
    \end{equation*}
    where the amplification functional $\Rscr(A,q)$ compares the variance against the $L^q$-deviation and the aggregated $q$-influences. This further leads to a quantum Talagrand-type isoperimetric inequality (Corollary~\ref{cor:quantum isoper}), a quantum Eldan--Gross inequality (Corollary~\ref{cor:quantum Eldan-Gross}), and a quantum Cordero-Erausquin--Eskenazis-type energy inequality (Corollary~\ref{cor:quantum-p-q-energy}).
    \item A \emph{high-order influence--variance inequality} (Theorem \ref{thm:influence-var}) that relates the high-order $p$-influence $\Inf^p_J[A]= \left\|d_J A\right\|_p^p$ to the partial variance functional $V_J(A)=\int_0^\infty2\left\|d_J P_t A\right\|_2^2 \dif t$: for all $A\in M_2(\mathbb{C})^{\otimes n}$, $J\subseteq [n]$ with $|J|=k\ge 1$, $p\in[1,2]$, and $q\in[1,2)$,
    \begin{equation*}
        \|A\|_\infty^{2-p}\,\Inf_J^p[A] \;\gtrsim_k\; V_J(A)\,\max\left\{k,\,
    \frac{q}{2(2-q)}\ln^{+}\!\left(\frac{k\,V_J(A)}{\Inf_J^q[A]^{2/q}}\right)\right\}.
    \end{equation*}
    While the high-order quantum Talagrand-type bounds of Blecher, Gao and Xu~\cite{Blecher2024geometricinfluencesquantumboolean} relate Fourier tail quantities to aggregated influences globally, our inequality provides a local subset-wise perspective on these principles. From this result, we further obtain quantum Talagrand--KKL-type high-order influence inequalities (Corollary~\ref{cor:quantum-p-q-influence}), and quantum high-order partial isoperimetric bounds (Corollary~\ref{cor:quantum-partial-isoper}).
\end{itemize}

Both of our main theorems arise from a unified two-step variance-decay framework. The first step (Lemmas~\ref{lem:quantum-energy-var-ineq} and~\ref{lem:quantum-high-order-influence-var-ineq}) establishes a general implication: for any individual observable, short-time variance decay along the depolarizing semigroup yields a corresponding Talagrand-type inequality. The second step (Lemmas~\ref{lem:quantum-energy-var-decay} and~\ref{lem:quantum-high-order-influence-var-decay}) estimates the rates of variance decay using hypercontractivity and log-convexity arguments.
This modular framework unifies first-order and high-order Talagrand-type phenomena through the variance-decay perspective, built on semigroup arguments of broad applicability.

The paper is organized as follows. Section~\ref{sec:preliminaries} introduces the basics of quantum Boolean analysis, including Fourier--Pauli analysis, the depolarizing semigroup, conditional expectations and derivatives, $\alpha$-gradients, and key functional inequalities. Section~\ref{sec:proofs-main} contains the proofs of our two main theorems. Section~\ref{sec:applications} derives the corollaries mentioned above; there, each corollary is accompanied by a brief recollection of the corresponding classical inequality for comparison. Several technical proofs and auxiliary lemmas are collected in the appendices.

\section{Quantum Boolean Analysis}\label{sec:preliminaries}

In this section, we develop the analytic framework for the quantum Boolean cube, emphasizing its parallels with the classical case and introducing the operator-valued structures that will be central to our results. We begin with the \emph{Fourier--Pauli expansion}, the quantum analogue of the Fourier--Walsh expansion on the classical cube, which allows us to decompose observables and express key quantities in terms of Fourier weights. We then turn to the \emph{depolarizing semigroup}, together with \emph{conditional expectations} and \emph{discrete derivatives}, which furnish the basic calculus underlying our analysis. Building on this, we introduce the \emph{$\alpha$-gradients} generalizing the \emph{carr\'e du champ}, and establish key analytic estimates underpinning our main results: the \emph{gradient bound}, \emph{Lipschitz smoothing}, and \emph{$L^p$-Poincar\'e} inequalities. Finally, we present the functionals of \emph{energies}, \emph{influences}, and \emph{partial variances}, which collectively form the analytic foundation for our Talagrand-type energy--variance and influence--variance inequalities.

\subsection{Quantum Boolean cube and Fourier--Pauli analysis}

As a noncommutative analogue of scalar functions on the classical Boolean cube, we consider \emph{observables} in the $n$-qubit matrix algebra $M_2(\Cbb)^{\otimes n}\cong M_{2^n}(\Cbb)$. Here, the expectation over the uniform measure is replaced by the \emph{normalized trace} $\tau:=2^{-n}\tr(\cdot)$. In analogy with the classical $L^p$ norms, we use the \emph{normalized Schatten-$p$ norms} on $ M_2(\mathbb{C})^{\otimes n}$:
\begin{equation*}
    \left\|A\right\|_p:=\tau (|A|^p)^{\frac{1}{p}} = \left(2^{-n}\sum_{k=1}^{2^n} \lambda_k(|A|)^p \right)^{\frac{1}{p}},\ 1\leq p<\infty,\quad \left\|A\right\|_\infty:= \max_k \lambda_k(|A|)
\end{equation*}
where $\{\lambda_k(|A|)\}_{k=1}^{2^n}$ denotes the eigenvalues of $|A|:=(A^\ast A)^{1/2}$. Correspondingly, we define the \emph{variance} of $A\in M_2(\mathbb{C})^{\otimes n}$ by $\Var(A):=\left\|A-\tau(A)\right\|_2^2$, where, by a convention used throughout the paper, we write $A-\tau(A)$ for $A-\tau(A)\1$, with $\1$ being the identity matrix in $M_2(\Cbb)^{\otimes n}$.

Following \cite[Definition 3.1]{MO2010A}, $A\in M_2(\mathbb{C})^{\otimes n}$ is called a \emph{quantum Boolean function} if $A$ is unitary Hermitian ($A=A^\ast$ and $A^2=\1$). A typical example is given by the Pauli matrices
\begin{equation*}
    \sigma_0=\begin{bmatrix}1&0\\0&1\end{bmatrix},\quad \sigma_1=\begin{bmatrix}0&1\\1&0\end{bmatrix},\quad
    \sigma_2=\begin{bmatrix}0&-i\\i&0\end{bmatrix},\quad \sigma_3=\begin{bmatrix}1&0\\0&-1\end{bmatrix}.
\end{equation*}
The classical observation \cite[Exercise 1.17]{Ryan2014book} for $\{-1,1\}$-valued functions analogously holds for quantum Boolean functions:
\begin{proposition}\label{prop:Boolean-L1-L2-var}
For every unitary Hermitian $A\in M_2(\mathbb{C})^{\otimes n}$, $\Var(A) = \|A-\tau(A)\|_{1}$.
\end{proposition}
\begin{proof}
Since $A$ is unitary and Hermitian, its spectrum consists of eigenvalues $\pm 1$. Thus we may write $A = \Pi_{+} - \Pi_{-}$, where $\Pi_{\pm}$ are the orthogonal projections onto the $\pm 1$-eigenspaces. Let $u:=\tau(\Pi_+)$, so that $\tau(\Pi_{-})=1-u$ because $\Pi_{+}+\Pi_{-}=\1$. Then
\begin{equation*}
    \tau(A) = \tau(\Pi_+) - \tau(\Pi_-) =2u-1 \implies |A-\tau(A)| = \left(2-2u\right) \Pi_+ + 2u\,\Pi_-.
\end{equation*}
Hence
\begin{align*}
    \left\|A-\tau(A)\right\|_{1} &= \tau(|A-\tau(A)|) = \left(2-2u\right) \tau(\Pi_+) + 2u\,\tau(\Pi_-)=4u\left(1-u\right)\\
    &= \left(2-2u\right)^2 \tau(\Pi_+) + \left(2u\right)^2 \tau(\Pi_-) = \tau(|A-\tau(A)|^2) = \Var(A),
\end{align*}
proving the claim.
\end{proof}

A basic feature of analysis on the classical Boolean cube is that every scalar function admits a Fourier--Walsh expansion in the character basis. In the quantum setting, the analogue of this is the \emph{Fourier--Pauli expansion} in the basis of tensor product Pauli matrices. More specifically, in the $n$-qubit algebra $M_2(\Cbb)^{\otimes n}$, an orthonormal basis with respect to the normalized Hilbert--Schmidt inner product $\langle X,Y\rangle := \tau(X^\ast Y)$ is given by
\begin{equation*}
    \sigma_s := \sigma_{s_1}\otimes\dots\otimes\sigma_{s_n}, \quad s=(s_1,\dots,s_n)\in\{0,1,2,3\}^n.
\end{equation*}
Thus every $A\in M_2(\Cbb)^{\otimes n}$ admits a unique Fourier--Pauli expansion
\begin{equation*}
    A=\sum_{s\in\{0,1,2,3\}^n}\widehat{A}_s\,\sigma_s,
    \qquad \widehat{A}_s:=\langle\sigma_s,A\rangle=\frac{1}{2^n}\tr(\sigma_s A).
\end{equation*}
In particular, if $A$ is Hermitian, then $\widehat{A}_s\in\R$ for all $s$.

Building on the Fourier--Pauli expansion, it is natural to organize the Fourier coefficients according to the degrees.
For $s\in\{0,1,2,3\}^n$, define its \emph{support} by $\supp(s):=\{j\in[n]: s_j\neq 0\}$, and its \emph{degree} (or \emph{level}) by $|\supp(s)|$. The \emph{Fourier weight} of $A\in M_2(\mathbb{C})^{\otimes n}$ at degree $d\in[n]$, as well as at degrees at least $d$, are then defined respectively by
\begin{equation}\label{eq:Fourier-weights}
    W^{=d}[A]:=\sum\limits_{\substack{s\in\{0,1,2,3\}^n\\ |\supp(s)|=d}}|\widehat{A}_s|^2,\quad W^{\ge d}[A]:= \sum_{m=d}^{n}W^{=m}[A].
\end{equation}
In this notation, the variance of $A$ can be expressed in terms of its Fourier weights:
\begin{equation}\label{eq:var-Fourier}
    \Var(A)=\sum\limits_{\substack{s\in\{0,1,2,3\}^n\\ s\neq 0}}|\widehat{A}_s|^2 = W^{\ge 1}[A].
\end{equation}

\subsection{Depolarizing semigroups, conditional expectations and derivatives}
The heat semigroup plays a central role in analysis on the classical Boolean cube, providing the framework for techniques from classical analysis. In the quantum setting, the natural analogue is the \emph{depolarizing semigroup} on $M_2(\Cbb)^{\otimes n}$ defined by
\begin{equation*}
P_t=e^{-t\mathcal L}:=\left(e^{-t}\mathbb{I}+(1-e^{-t})\frac 1 2\tr(\cdot) \1\right)^{\otimes n},\quad t\geq 0,
\end{equation*}
which is an {\em ergodic, tracially symmetric quantum Markov semigroup}. Here $\mathbb{I}$ denotes the identity map on $M_2(\mathbb{C})$. The semigroup $P_t$, and hence its generator $\Lcal$, are diagonal in the Pauli basis:
\begin{equation}\label{eq:semigroup-Fourier}
    P_t A=\sum_{s\in \{0,1,2,3\}^n}e^{-t\lvert \supp(s)\rvert}\widehat A_s\sigma_s, \quad \Lcal A = \sum_{s\in \{0,1,2,3\}^n}\left|\supp(s)\right|\widehat A_s\sigma_s.
\end{equation}
The classical functional inequalities extend naturally to the quantum setting~\cite{MO2010A}:
\begin{proposition}[Poincar\'{e} inequality and variance decay]
\label{prop:quantum-Poincare}
For every $A\in M_2(\Cbb)^{\otimes n}$,
\begin{equation}\label{eq:Poincare-ineq}
    \Var(A)\le \Ecal[A]:=\tup{A,\Lcal A},
\end{equation}
and equivalently, for all $t\geq 0$,
\begin{equation}\label{eq:Poincare-var-decay}
    \Var(P_t A)\le e^{-2t}\Var(A).
\end{equation}
\end{proposition}

\begin{proposition}[Log-Sobolev inequality and hypercontractivity]\label{prop:quantum-hyper}
For every $A\in M_2(\Cbb)^{\otimes n}$,
\begin{equation}\label{eq:log-Sobolev-ineq}
    \operatorname{Ent}[|A|^2]:= \tau\bigl(|A|^2 \ln(|A|^2)\bigr) -  \tau\bigl(|A|^2\bigr) \ln\bigl(\tau(|A|^2)\bigr) \leq 2\Ecal[A],
\end{equation}
and equivalently, for all $t\geq 0$,
\begin{equation}\label{eq:hypercontractivity}
    \|P_t A\|_{2}\leq \|A\|_{1+e^{-2t}}.
\end{equation}
\end{proposition}

As a tensor product semigroup, $P_t$ has the decomposition:
\begin{equation}\label{eq:semigroup-decomposition}
    P_t=\prod_{j=1}^n P_t^j,\quad P_t^{j}=\tau_j+e^{-t}d_j,
\end{equation}
and hence $\Lcal=\sum_{j=1}^n d_j$, where $\tau_j$ and $d_j$ are, respectively, the conditional expectation and the discrete derivative with respect to the $j$-th coordinate, defined by
\begin{equation*}
    \tau_j:=\mathbb{I}^{\otimes (j-1)}\otimes\, \frac 1 2\tr(\cdot) \1 \,\otimes \mathbb{I}^{\otimes(n-j)},\quad d_j:=\mathbb{I}^{\otimes n} - \tau_j = \mathbb{I}^{\otimes(j-1)}\otimes \left(\mathbb{I}-\frac 1 2\tr(\cdot) \1 \right)\otimes \mathbb{I}^{\otimes (n-j)}.
\end{equation*}
These are noncommutative analogues of the conditional expectation $\E_i$ and the discrete derivative $D_i$ on the classical Boolean cube.
Note that $\tau_j$'s and $d_j$'s are projections on $M_2(\Cbb)^{\otimes n}$ that commute with each other. Define the \emph{high-order} conditional expectation and discrete derivative with respect to nonempty $J\subseteq[n]$ by
\begin{equation*}
    \tau_J:=\prod_{j\in J}\tau_j,\quad d_J:=\prod_{j\in J} d_j.
\end{equation*}
Since $\tau_j$ is the projection onto the subalgebra independent of the $j$-th coordinate, the Fourier--Pauli expansion of $\tau_J$ and $d_J$ are, respectively, given by
\begin{equation}\label{eq:exp-der-Fourier}
    \tau_{J}(A) = \sum_{\substack{s\in \{0,1,2,3\}^n\\\supp(s)\subseteq J^{c}}}\widehat A_s \sigma_s, \quad d_{J} A = \sum_{\substack{s\in \{0,1,2,3\}^n\\\supp(s)\supseteq J}}\widehat A_s \sigma_s.
\end{equation}
Using the tensor product decomposition \eqref{eq:semigroup-decomposition} of $P_t$, we obtain:
\begin{lemma}\label{lem:j-derivative-decay}
Let $A\in M_2(\Cbb)^{\otimes n}$, $t\ge 0$, $J\subseteq[n]$ with $|J|=k\geq 1$, and $p \in [1,\infty]$. Then
\begin{equation*}
    \|d_JP_tA\|_p\le e^{-kt}\,\|d_JA\|_p.
\end{equation*}
In particular, for $k=1$ we have $\|d_jP_tA\|_p\le e^{-t}\|d_jA\|_p$.
\end{lemma}
\begin{proof}
Fix $t\geq 0$ and $J\subseteq[n]$ with $|J|=k\geq 1$. By \eqref{eq:semigroup-decomposition}, we can factorize $P_t=P^J_t P^{-J}_t$ with
\begin{equation*}
    P_t^J:=\prod_{j\in J} P_t^j,\quad P_t^{-J}:=\prod_{j\notin J} P_t^j,\quad P_t^{j}=\tau_j+e^{-t}d_j.
\end{equation*}
By commutation and the fact that $d_j P_t^j=e^{-t}d_j$, we obtain
\begin{equation*}
    d_JP_t=P^{-J}_t \prod_{i\in J}(d_i P_t^i) = e^{-kt}P^{-J}_t d_J.
\end{equation*}
Further note that $P^{-J}_t$ is also a tracially symmetric quantum Markov semigroup, which is unital and completely positive, and hence $L^p$-contractive. The result follows.
\end{proof}

Moreover, the derivatives enjoy a ``regularity'' under conditional expectations:
\begin{proposition}
\label{prop:derivative-regularity}
For every $A\in M_2(\Cbb)^{\otimes n}$ and $J\subseteq [n]$ with $|J|=k\geq 1$, we have
\begin{equation*}
    |d_J A|^2 \leq 3^{k}\tau_J(|d_J A|^2).
\end{equation*}
In particular, taking $k=1$ yields $|d_j A|^2 \leq 3 \Var_{j}(A)$, where $\Var_j(A):=\tau_j\bigl(|A-\tau_j(A)|^2\bigr)$.
\end{proposition}
\begin{remark}
When $k=1$, the analogue of Proposition \ref{prop:derivative-regularity} on the classical Boolean cube is the pointwise identity $|D_if|^2 = \E_i(|D_if|^2)$, since $|D_if|$ depends only on coordinates other than the $j$-th. In the quantum case, a Bell-state example shows that the constant $3$ is sharp. Consider $A=|\Phi^+\rangle\langle\Phi^+| \in  M_2(\Cbb)^{\otimes 2}$, where
\begin{equation*}
    |\Phi^+\rangle=\frac{1}{\sqrt2}(|0\rangle \otimes |0\rangle +|1\rangle \otimes |1\rangle),\quad
    |0\rangle :=
    \begin{bmatrix}
      1 \\
      0
    \end{bmatrix}
    ,\
    |1\rangle :=
    \begin{bmatrix}
      0 \\
      1
    \end{bmatrix}
    .
\end{equation*}
Then $\tau_1(A)= \frac{1}{4} (|0\rangle\langle0| + |1\rangle\langle1|) = \frac14 \1$, and hence
\begin{equation*}
    |d_1A|^2=\bigl|A-\tau_1(A)\bigr|^2=\frac12 A+\frac1{16}\1, \qquad \Var_{1}(A) = \tau_1(|d_1A|^2)=\frac{3}{16}\1.
\end{equation*}
Now testing $3\Var_{1}(A) - |d_1A|^2$ on $|\Phi^+\rangle$ gives
\begin{equation*}
    \langle\Phi^+| \left(3\Var_{1}(A) - |d_1A|^2 \right) |\Phi^+\rangle = \frac{9}{16} - \frac{1}{2} \langle\Phi^+|\, A\, |\Phi^+\rangle - \frac1{16} = 0,
\end{equation*}
which shows that the constant $3$ is attained and cannot be improved.
\end{remark}
\begin{proof}
Let $A\in M_2(\Cbb)^{\otimes n}$ and $J\subseteq [n]$ with $|J|=k\geq 1$. Organizing the Fourier--Pauli expansions according to the $J$-coordinates gives
\begin{equation*}
    A = \sum_{x\in \{0,1,2,3\}^{J}} A^{J}_{x}, \quad d_{J} A = \sum_{x\in \{1,2,3\}^{J}} A^{J}_{x},\quad  A^{J}_{x} := \sum_{s:s^J=x} \widehat{A}_s \sigma_{s},
\end{equation*}
where $s^{J}$ denotes the $J$-coordinates of $s$. Hence 
\begin{equation*}
    |d_J A|^{2} = \sum_{x,y\in \{1,2,3\}^{J}} (A^J_x)^{\ast} A^J_y,\quad \tau_J(|d_J A|^2) = \sum_{x\in \{1,2,3\}^{J}}  |A^J_x|^2.
\end{equation*}
To bound the cross-terms, expanding $|A^J_x -A^J_y|^2\geq 0$ yields
\begin{equation*}
    (A^J_x)^{\ast} A^J_y + (A^J_y)^{\ast} A^J_x \leq  |A^J_x|^2 + |A^J_y|^2.
\end{equation*}
Writing the expansion of $|d_J A|^{2}$ in symmetric sum gives
\begin{align*}
    2|d_J A|^{2} = \sum_{x,y\in \{1,2,3\}^{J}} \left[(A^J_x)^{\ast}  A^J_y+ (A^J_y)^{\ast} A^J_x\right]\leq \sum_{x,y\in \{1,2,3\}^{J}} \left[|A^J_x|^2 + |A^J_y|^2 \right] = 2\cdot 3^{k}\tau_J(|d_J A|^2).
\end{align*}
Canceling constant $2$ on both sides completes the proof.
\end{proof}

Using Proposition \ref{prop:derivative-regularity}, we can improve the trivial estimate $\left\|d_J\right\|_{\infty\to\infty}\le 2^{|J|}$:
\begin{proposition}\label{prop:derivative-infty-norm}
For every $J\subseteq[n]$ with $|J|=k\geq 1$, we have $\left\|d_J\right\|_{\infty\to\infty}\le \sqrt{3}^{k}$.
\end{proposition}
\begin{remark}
On the classical Boolean cube, due to the pointwise identity $|D_if|^2 = \E_i(|D_if|^2)$, we have an order-independent bound $\left\|D_J\right\|_{\infty\to\infty}\le 1$.
\end{remark}
\begin{proof}
Using the expansion according to $J$-coordinates, it follows clearly that
\begin{equation*}
    \tau_{J}(|d_{J} A|^{2}) = \sum_{x\in \{1,2,3\}^{J}}  |A^J_x|^2 \leq \sum_{x\in \{0,1,2,3\}^{J}}  |A^J_x|^2 = \tau_{J}(|A|^{2}).
\end{equation*}
Applying Proposition \ref{prop:derivative-regularity} yields
\begin{equation*}
    |d_{J} A|^{2} \leq 3^{k} \tau_{J}(|d_{J} A|^{2}) \leq 3^{k} \tau_{J}(|A|^{2}).
\end{equation*}
Further by the $L^\infty$-contractivity of the conditional expectation $\tau_J$, 
\begin{equation*}
    \left\|d_J A \right\|_\infty^{2} \leq 3^{k} \left\|\tau_{J}(|A|^{2})\right\|_\infty \leq 3^{k} \left\|A\right\|_\infty^{2}.
\end{equation*}
Taking square root on both sides completes the proof.
\end{proof}

\subsection{Carr\'e du champ and gradients}
\label{sec:gradients}
A key notion in semigroup analysis is the carr\'{e} du champ operator, which generalizes the square of gradients. For a tracially symmetric quantum Markov semigroup $P_t = e^{-t\Lcal}$, the associated carr\'{e} du champ operator is the nonnegative sesquilinear map given by
\begin{equation*}
    \Gamma(A,B) := \frac{1}{2}\left[(\Lcal A)^* B + A^* (\Lcal B) - \Lcal (A^* B)\right].
\end{equation*}
Its quadratic specialization $\Gamma(A):= \Gamma(A,A)\geq 0$ serves as the analogue of the squared gradient. The definition of the carr\'{e} du champ yields the following fundamental identity:
\begin{equation}\label{eq:grad-identity}
    P_t(|A|^2) - |P_t A|^2 = \int_0^t 2 P_s \Gamma(P_{t-s} A) \dif s.
\end{equation}

For the depolarizing semigroup $P_t=e^{-t\Lcal}$ on $M_2(\Cbb)^{\otimes n}$, recall that its generator decomposes as $\Lcal=\sum_{j=1}^n d_j$. Accordingly, the associated carr\'{e} du champ admits the natural splitting
\begin{equation*}
    \Gamma(A)=\sum_{j=1}^n \Gamma_j(A),\quad 2\Gamma_j(A):= d_j(A^\ast)A + A^\ast d_j(A) - d_j(A^\ast A)=\Var_j(A)+|d_jA|^2,
\end{equation*}
where each local term $\Gamma_j(A)$ lies exactly halfway between the conditional variance $\Var_j(A)$ and the pointwise squared derivative $|d_jA|^2$. This midpoint formulation naturally leads to a family of \emph{$\alpha$-gradients}:
\begin{equation}\label{def:alpha-gradient}
    |\nabla^{\alpha}A|^2 := \sum_{j=1}^{n} |\nabla^{\alpha}_{j} A|^2, \quad |\nabla^{\alpha}_{j} A|^2:= \left(1-\alpha\right)\Var_j(A) + \alpha\,|d_j A|^2,\quad \alpha\in [0,1],
\end{equation}
which interpolate between the conditional variances and the squared derivatives, and reproduce the carr\'{e} du champ when $\alpha=1/2$. In particular,
\begin{equation*}
    |\nabla A| := |\nabla^{1} A| = \left(\sum_{j=1}^{n} |d_jA|^2\right)^{1/2}
\end{equation*}
gives the canonical discrete gradient corresponding to $d_j$, $j\in [n]$. Moreover, since $\tau\circ\tau_j=\tau$, the $\alpha$-gradients share the same $L^2$-norm:
\begin{equation*}
    \bigl\||\nabla_j^{\alpha}A|\bigr\|_2^2 = \left(1-\alpha\right)\tau(\Var_j(A)) + \alpha\,\tau(|d_j A|^2) = \left\|d_j A\right\|_2^2.
\end{equation*}

Using Proposition~\ref{prop:derivative-regularity}, we can compare the $\alpha$-gradients for different values of $\alpha$:
\begin{lemma}[Gradient comparison]
\label{lem:grad-compare}
Let $A\in M_2(\Cbb)^{\otimes n}$, $j\in[n]$ and $\alpha,\beta\in [0,1]$. Then
\begin{equation}\label{ineq:bounding with two parameters}
    |\nabla_j^{\alpha}A|^{2}\geq B^{\text{\rm\ref{lem:grad-compare}}}(\alpha,\beta) |\nabla_j^{\beta}A|^{2},
\end{equation}
and consequently, $|\nabla^\alpha A|^2\ge B^{\text{\rm\ref{lem:grad-compare}}}(\alpha,\beta)\,|\nabla^\beta A|^2$, where
\begin{equation*}
    B^{\text{\rm\ref{lem:grad-compare}}}(\alpha,\beta) =
    \begin{cases}
        \frac{1+2\alpha}{1+2\beta}, & \alpha\leq \beta,\\
        \frac{1-\alpha}{1-\beta}, & \alpha\geq \beta.
    \end{cases}
\end{equation*}
In particular, taking $\beta=1$ yields $|\nabla^\alpha A|\geq \bigl(\frac{1+2\alpha}{3}\bigr)^{1/2} |\nabla A|$.
\end{lemma}
\begin{remark}
On the classical Boolean cube, one has the pointwise identity $\Var_{i}(f)= |D_i f|^2$. Accordingly, the analogue of Lemma~\ref{lem:grad-compare} reduces to the trivial identity
$|\nabla_i^\alpha f|^2=|D_i f|^2$, so that $|\nabla^\alpha f|=|\nabla f|$ for all $\alpha\in[0,1]$.
\end{remark}
\begin{proof}
Let $A\in M_2(\Cbb)^{\otimes n}$ and $j\in[n]$. Applying Proposition \ref{prop:derivative-regularity} with $k=1$ yields
\begin{equation*}
    0\leq |d_j A|^2 \leq 3\Var_j(A).
\end{equation*}
\begin{itemize}
    \item If $\alpha\geq \beta$, then $\frac{\alpha}{1-\alpha}\geq \frac{\beta}{1-\beta}$, and hence
    \begin{align*}
        |\nabla^{\alpha}_{j} A|^2 &=  (1-\alpha) \Var_j(A) + \alpha |d_j A|^2 =  \frac{1-\alpha}{1-\beta} \left((1-\beta) \Var_j(A) + \alpha \frac{1-\beta}{1-\alpha}  |d_j A|^2\right)\\
        &\geq \frac{1-\alpha}{1-\beta} \left((1-\beta) \Var_j(A) + \beta |d_j A|^2\right) = \frac{1-\alpha}{1-\beta}\, |\nabla^{\beta}_{j} A|^2.
    \end{align*}
    \item If $\alpha\leq \beta$, then
    \begin{align*}
        |\nabla^{\alpha}_{j} A|^2 &= (1-\alpha) \Var_j(A) + \alpha |d_j A|^2 \\
        &\geq \left(1-\alpha-3\frac{\beta-\alpha}{1+2\beta}\right) \Var_j(A) + \left(\alpha + \frac{\beta-\alpha}{1+2\beta}\right) |d_j A|^2 \\
        &=\frac{1+2\alpha}{1+2\beta} \left((1-\beta) \Var_j(A) + \beta |d_j A|^2\right) = \frac{1+2\alpha}{1+2\beta} \, |\nabla^{\beta}_{j} A|^2.
    \end{align*}
\end{itemize}
Combining the estimates in respective cases completes the proof.
\end{proof}

Gradient bound estimates describe how gradients behave under the semigroup flow, forming a cornerstone of the  Bakry--\'{E}mery calculus. A complete study of such estimates for general quantum Markov semigroups was carried out in \cite{Wirth2021complete}. In the quantum Boolean setting, applying \eqref{eq:semigroup-decomposition} together with Lemma~\ref{lem:grad-compare}, we establish gradient bound estimates for the $\alpha$-gradients under the depolarizing semigroup, refining the standard bound
\begin{equation}\label{eq:grad-estimate-w2021}
    \Gamma(P_t A) \leq e^{-t} P_t\Gamma(A)
\end{equation}
provided in \cite[Example 1]{Wirth2021complete} (see also the restatement in \cite[Lemma 2.4]{Rouze2024quantum}).
\begin{proposition}[Gradient bound]\label{prop:Gradient-estimate}
For $A\in M_2(\Cbb)^{\otimes n}$, $j\in [n]$, $\alpha\in [0,1]$, and $t\geq 0$,
\begin{equation*}
    |\nabla^{\alpha}_j P_t A|^{2} \leq C^{\text{\rm\ref{prop:Gradient-estimate}}}_{\alpha}(t) P_t (|\nabla^{\alpha}_j A|^{2}),
\end{equation*}
and consequently, $|\nabla^{\alpha} P_t A|^{2} \leq  C^{\text{\rm\ref{prop:Gradient-estimate}}}_{\alpha}(t)  P_t (|\nabla^{\alpha} A|^{2})$, where $C^{\text{\rm\ref{prop:Gradient-estimate}}}_{\alpha}(t):=\frac{1+2\alpha}{e^{2t} + 2\alpha e^{t}}$.
\end{proposition}
\begin{remark}
Here $C^{\text{\rm\ref{prop:Gradient-estimate}}}_{\alpha}(t)\leq e^{-\frac{2+2\alpha}{1+2\alpha}t}$. In particular, taking $\alpha=1/2$ yields
\begin{equation*}
    \Gamma(P_t A) \leq \frac{2}{e^{2t} + e^{t}} P_t \Gamma(A) \leq e^{-\frac{3}{2}t}P_t \Gamma(A),
\end{equation*}
which provides a sharper decay rate than \eqref{eq:grad-estimate-w2021}. Moreover, the gradient bound estimate on the classical Boolean cube is given by $\Gamma(P_t f) \leq e^{-2t} P_t\Gamma(f)$.
\end{remark}
\begin{proof}
Let $A\in M_2(\Cbb)^{\otimes n}$, $j\in[n]$, $\alpha\in[0,1]$, and $t\geq 0$. By \eqref{eq:semigroup-decomposition}, we can factorize $P_t= P^j_t P^{-j}_t$ with $P^{-j}_t := \prod_{i\neq j} P^i_t$, where $P^i_t = \tau_i+e^{-t}d_i$, for $i\in [n]$. We estimate the behavior of $|\nabla_j^{\alpha} \cdot|$ under $P^j_t$ and $P^{-j}_t$ respectively:
\begin{itemize}
    \item We first prove that $|\nabla^{\alpha}_j P^j_t A|^{2} \leq C^{\text{\rm\ref{prop:Gradient-estimate}}}_{\alpha}(t)\, P^j_t (|\nabla^{\alpha}_j A|^{2})$. Since $d_j P_t^j=e^{-t}d_j$, we have
    \begin{align*}
        |\nabla^{\alpha}_j P^j_t A|^{2} &= \left(1-\alpha\right)\tau_j(|d_jP_t^jA|^2) + \alpha\, |d_jP_t^jA|^2\\
        &= e^{-2t} \left(\left(1-\alpha\right)\tau_j(|d_j A|^2) + \alpha\, |d_j A|^2\right) = e^{-2t} |\nabla^{\alpha}_j A|^{2}.
    \end{align*}
    On the other hand, applying Lemma~\ref{lem:grad-compare} with $\tau_{j}(|\nabla^{\alpha}_j A|^{2}) = \Var_j(A) = |\nabla^{0}_j A|^{2}$ yields
    \begin{align*}
        P_t^j(|\nabla^{\alpha}_j A|^{2})& = e^{-t}|\nabla^{\alpha}_j A|^{2}+\left(1-e^{-t}\right)\tau_j(|\nabla^{\alpha}_j A|^{2})\\
        &\geq \left(e^{-t} + B^{\text{\rm\ref{lem:grad-compare}}}(0,\alpha) \right) |\nabla^{\alpha}_j A|^{2} =  \frac{1+2\alpha e^{-t}}{1+2\alpha}\,|\nabla^{\alpha}_j A|^{2}.
    \end{align*}
    Combining the bounds we obtain
    \begin{equation*}
        |\nabla^{\alpha}_j P^j_t A|^{2}=e^{-2t}|\nabla^{\alpha}_j A|^{2}\le e^{-2t}\frac{1+2\alpha}{1+2\alpha e^{-t}}\, P_t^j (|\nabla^{\alpha}_j A|^{2}) = C^{\text{\rm\ref{prop:Gradient-estimate}}}_{\alpha}(t)\,  P^j_t (|\nabla^{\alpha}_j  A|^{2})
    \end{equation*}
    as desired.
    \item We then prove that $|\nabla^{\alpha}_j P^{-j}_t A|^{2} \leq P^{-j}_t (|\nabla^{\alpha}_j A|^{2})$. Note that $P^{-j}_t$ commutes with $\tau_{j}$ and $d_{j}$, and is unital and completely positive. By the Kadison--Schwarz inequality (Lemma~\ref{lem:Kadison--Schwarz}),
    \begin{equation*}
        |d_j P^{-j}_t A|^{2} = |P^{-j}_t d_j  A|^{2} \leq P^{-j}_t(|d_j  A|^{2}).
    \end{equation*}
    Hence we obtain
    \begin{align*}
        |\nabla^{\alpha}_j P^{-j}_t A|^{2} &= \left(1-\alpha\right)\tau_j(|d_j P^{-j}_t A|^{2}) + \alpha\,|d_j P^{-j}_t A|^2\\
        &\leq \left(1-\alpha\right)P^{-j}_t \tau_j(|d_j  A|^{2}) + \alpha\, P^{-j}_t(|d_j  A|^2) =  P^{-j}_t(|\nabla^{\alpha}_j A|^{2})
    \end{align*}
    as desired.
\end{itemize}
Finally, combining the estimates yields
\begin{equation*}
    |\nabla^{\alpha}_j P_t A|^{2} = |\nabla^{\alpha}_j P_t^j P^{-j}_t A|^{2}\leq C^{\text{\rm\ref{prop:Gradient-estimate}}}_{\alpha}(t)\, P_t^j (|\nabla^{\alpha}_j P^{-j}_t A|^{2}) \leq  C^{\text{\rm\ref{prop:Gradient-estimate}}}_{\alpha}(t)\, P_t (|\nabla_j^{\alpha}A|^2),
\end{equation*}
which completes the proof.
\end{proof}

Combining the gradient bound estimates in Proposition~\ref{prop:Gradient-estimate} with the fundamental identity \eqref{eq:grad-identity}, we obtain the following estimate. Here, the case $q=\infty$ is the sharpest in our approach and precisely expresses a dimension-free $L^\infty$-to-Lipschitz smoothing property of the depolarizing semigroup, which serves as a key input in the proof of Theorem~\ref{thm:energy-variance}.
\begin{proposition}[Smoothing property]\label{prop:Lipschitz-smoothing}
For $A\in M_2(\Cbb)^{\otimes n}$, $\alpha\in [0,1]$, $q\in [2,\infty]$, and $t> 0$,
\begin{equation}
    \left\||\nabla^{\alpha} P_t A| \right\|_{q} \leq G^{\text{\rm\ref{prop:Lipschitz-smoothing}}}_{\alpha}(t)^{-1/2} \|A\|_{q},
\end{equation}
where
\begin{equation*}
G^{\text{\rm\ref{prop:Lipschitz-smoothing}}}_{\alpha}(t):=
\begin{cases}
\dfrac{(e^t-1)(e^t+1+4\alpha)}{2(1-\alpha)(1+2\alpha)}, & \alpha\le 1/2,\\[3mm]
\dfrac{(e^t-1)(e^t+3)}{1+2\alpha}, & \alpha\ge 1/2.
\end{cases}
\end{equation*}
\end{proposition}
\begin{remark}
In particular, taking $\alpha=1$ and $q=\infty$ yields
\begin{equation}\label{ineq:Lipschitz-smoothing}
    \left\||\nabla P_t A| \right\|_{\infty} \leq G^{\text{\rm\ref{prop:Lipschitz-smoothing}}}_{1}(t)^{-1/2} \left\|A\right\|_{\infty},
\end{equation}
or equivalently, $\left\|P_t\right\|_{\infty\to\text{Lip}} \leq G^{\text{\rm\ref{prop:Lipschitz-smoothing}}}_{1}(t)^{-1/2}$, with $\left\|\,\cdot\,\right\|_{\text{Lip}}:= \left\||\nabla \cdot|\right\|_{\infty}$ the Lipschitz seminorm. Systematic applications of such Lipschitz smoothing property in classical functional and geometric inequalities can be found in \cite{de2025properties}. Moreover, the same approach applied on the classical Boolean cube yields the Lipschitz smoothing estimate $\left\||\nabla P_t f| \right\|_{\infty} \leq (e^{2t}-1)^{-1/2} \left\|f\right\|_{\infty}$.
\end{remark}
\begin{proof}
Let $A\in M_2(\Cbb)^{\otimes n}$, $\alpha\in [0,1]$ and $t>0$. We first prove that 
\begin{equation*}
    P_t(|A|^2) \geq G^{\text{\rm\ref{prop:Lipschitz-smoothing}}}_{\alpha}(t)\, |\nabla^{\alpha} P_{t} A|^{2}.
\end{equation*}
The carr\'{e} du champ identity \eqref{eq:grad-identity} yields
\begin{equation*}
    P_t(|A|^2) \geq P_t(|A|^2) - |P_t A|^2 = \int_0^t 2 P_s \Gamma(P_{t-s} A) \dif s.
\end{equation*}
For $\beta\in [0,1]$ and $s\in [0,t]$, by Proposition~\ref{prop:Gradient-estimate} and Lemma~\ref{lem:grad-compare}, we deduce
\begin{align*}
    P_s \Gamma(P_{t-s} A) &\geq B^{\text{\rm\ref{lem:grad-compare}}}(1/2,\beta)\, P_s (|\nabla^{\beta} P_{t-s} A|^{2}) \geq B^{\text{\rm\ref{lem:grad-compare}}}(1/2,\beta)\, C^{\text{\rm\ref{prop:Gradient-estimate}}}_\beta(s)^{-1} |\nabla^\beta P_s P_{t-s} A|^2\\
    &= B^{\text{\rm\ref{lem:grad-compare}}}(1/2,\beta)\, C^{\text{\rm\ref{prop:Gradient-estimate}}}_\beta(s)^{-1} |\nabla^\beta P_tA|^2 \geq B^{\text{\rm\ref{lem:grad-compare}}}(1/2,\beta)\, C^{\text{\rm\ref{prop:Gradient-estimate}}}_\beta(s)^{-1}\,B^{\text{\rm\ref{lem:grad-compare}}}(\beta,\alpha) |\nabla^\alpha P_tA|^2.
\end{align*}
Note that
\begin{equation*}
    \sup_{\beta\in [0,1]} B^{\text{\rm\ref{lem:grad-compare}}}(1/2,\beta)\, C^{\text{\rm\ref{prop:Gradient-estimate}}}_\beta(s)^{-1}B^{\text{\rm\ref{lem:grad-compare}}}(\beta,\alpha) \geq 
    \begin{cases}
        B^{\text{\rm\ref{lem:grad-compare}}}(1/2,\alpha)\, C^{\text{\rm\ref{prop:Gradient-estimate}}}_\alpha(s)^{-1} = \frac{e^{2t} + 2\alpha e^{t}}{{2(1-\alpha)(1+2\alpha)}}, & \alpha\leq 1/2, \\[2mm]
        C^{\text{\rm\ref{prop:Gradient-estimate}}}_{1/2}(s)^{-1} B^{\text{\rm\ref{lem:grad-compare}}}(1/2,\alpha) = \frac{e^{2t} + e^{t}}{1+2\alpha},  & \alpha\geq 1/2.
    \end{cases}
\end{equation*}
Combining the bounds, and taking supremum over $\beta\in [0,1]$, we obtain
\begin{equation*}
    P_t(|A|^2) \geq \int_0^t 2 P_s \Gamma(P_{t-s} A) \dif s \geq  G^{\text{\rm\ref{prop:Lipschitz-smoothing}}}_{\alpha}(t)\, |\nabla^{\alpha} P_{t} A|^{2},
\end{equation*}
with
\begin{equation*}
    G^{\text{\rm\ref{prop:Lipschitz-smoothing}}}_{\alpha}(t) =
    \begin{cases}
        \displaystyle\int_0^t 2\, \frac{e^{2t} + 2\alpha e^{t}}{{2(1-\alpha)(1+2\alpha)}} \dif s = \frac{(e^t-1)(e^t+1+4\alpha)}{2(1-\alpha)(1+2\alpha)}, & \alpha\leq 1/2, \\[4mm]
        \displaystyle\int_0^t 2\, \frac{e^{2t} + e^{t}}{1+2\alpha} \dif s = \frac{(e^t-1)(e^t+3)}{1+2\alpha},  & \alpha\geq 1/2.
    \end{cases}
\end{equation*}
as desired.

Finally, for $q\in [2,\infty]$, by the contraction property of $P_t$ under the $q/2$-norm, we obtain
\begin{align*}
    \left\||\nabla^{\alpha} P_t A|\right\|_{q}^2 = \left\||\nabla^{\alpha} P_t A|^2\right\|_{q/2} &\leq G^{\text{\rm\ref{prop:Lipschitz-smoothing}}}_{\alpha}(t)^{-1} \left\|P_t (|A|^2)\right\|_{q/2}\\
    &\leq G^{\text{\rm\ref{prop:Lipschitz-smoothing}}}_{\alpha}(t)^{-1} \left\||A|^2\right\|_{q/2} =  G^{\text{\rm\ref{prop:Lipschitz-smoothing}}}_{\alpha}(t)^{-1} \left\|A\right\|_{q}^2,
\end{align*}
which completes the proof.
\end{proof}

Using a canonical duality argument (see, e.g., \cite{Ledoux2004SpectralGL,Ivanisvili2019improvingconstantendpointpoincare}), we derive the following $L^p$-Poincar\'e inequality from Proposition~\ref{prop:Lipschitz-smoothing}. When $p=1$, our result yields a sharper constant compared to the standard $L^1$-Poincar\'{e} inequality
\begin{equation}\label{eq:Rouze-quantum-L1-Poincare}
    \bigl\|\sqrt{\Gamma(A)}\bigr\|_{1} \geq \frac{1}{\sqrt{2}\pi} \left\|A-\tau(A)\right\|_{1}
\end{equation}
provided in \cite[Theorem 6.14]{Rouze2024quantum}.
\begin{proposition}[$L^p$-Poincar\'{e} inequality]
\label{prop:quantum-p-Poincare}
For all $A\in M_2(\Cbb)^{\otimes n}$ and $p\in[1,2]$,
\begin{equation*}
    \left\||\nabla A|\right\|_{p} \geq \frac{3}{2\pi} \left\|A-\tau(A)\right\|_{p}.
\end{equation*}
\end{proposition}
\begin{remark}
Using Lemma \ref{lem:grad-compare}, we further obtain $L^p$-Poincar\'{e} estimates for the $\alpha$-gradients:
\begin{equation*}
    \left\||\nabla^\alpha A|\right\|_{p} \geq \left(\frac{1+2\alpha}{3}\right)^{\frac{1}{2}} \left\||\nabla A|\right\|_{p} \geq \frac{3}{2\pi} \left(\frac{1+2\alpha}{3}\right)^{\frac{1}{2}}\left\|A-\tau(A)\right\|_{p}.
\end{equation*}
When $p=1$ and $\alpha=1/2$, this improves the constant in \eqref{eq:Rouze-quantum-L1-Poincare}. Moreover, the same approach applied on the classical Boolean cube yields the result with constant $2/\pi$ (see, e.g., \cite{Ivanisvili2019improvingconstantendpointpoincare}).
\end{remark}
\begin{proof}
Let $A\in M_2(\Cbb)^{\otimes n}$ and $p\in[1,2]$. We first prove that for all $t>0$,
\begin{equation*}
    \left\|A-P_tA\right\|_{p} \leq \left\||\nabla A|\right\|_p \int_{0}^{t}G^{\text{\rm\ref{prop:Lipschitz-smoothing}}}_{1}(s)^{-1/2}  \dif s.
\end{equation*}
For $B\in M_2(\Cbb)^{\otimes n}$, by $A-P_tA=\int_0^t \Lcal P_sA \dif s$ and the projection property of $d_j$,
\begin{equation*}
    \tup{A-P_tA, B}=\int_0^t \tup{\Lcal P_sA, B}\dif s = \int_0^t \sum_{j=1}^n \tup{d_j P_sA, B}\dif s = \int_0^t \sum_{j=1}^n \tup{d_j P_sA, d_j B}\dif s.
\end{equation*}
Since $P_s$ is $\tau$-symmetric and commutes with $d_j$, $\tup{d_j P_sA, d_j B} = \tup{d_j A, d_j P_s B}$, and hence
\begin{equation*}
     \tup{A-P_tA, B} = \int_0^t \sum_{j=1}^n \tup{d_j A, d_j P_s B} \dif s = \int_0^t \tau\bigg(\sum_{j=1}^n(d_jA)^*d_j(P_sB)\bigg) \dif s.
\end{equation*}
By the noncommutative H\"older inequality (Lemma~\ref{lem:nc-holder}), 
\begin{equation*}
    \left|\tau\bigg(\sum_{j=1}^n(d_jA)^*d_j(P_sB)\bigg)\right|\le \left\|\bigg(\sum_{j=1}^n |d_j A|^2\bigg)^{1/2}\right\|_p \left\|\bigg(\sum_{j=1}^n |d_jP_s B|^2\bigg)^{1/2}\right\|_{p'} = \left\||\nabla A|\right\|_p \left\||\nabla P_sB|\right\|_{p'},
\end{equation*}
where $p'=\frac{p}{p-1}\in[2,\infty)$. Applying Proposition~\ref{prop:Lipschitz-smoothing} with $\alpha=1$, we deduce
\begin{align*}
    \left|\tup{A-P_tA, B}\right| &\leq \int_0^t \left|\tau\bigg(\sum_{j=1}^n(d_jA)^*d_j(P_sB)\bigg)\right| \dif s \\
    &\leq \int_0^t \left\||\nabla A|\right\|_p \left\||\nabla P_sB|\right\|_{p'} \dif s \leq \left\||\nabla A|\right\|_p \left\||\nabla B|\right\|_{p'} \int_{0}^{t}G^{\text{\rm\ref{prop:Lipschitz-smoothing}}}_{1}(s)^{-1/2}  \dif s.
\end{align*}
Now by the norm duality, we obtain 
\begin{equation*}
    \left\|A-P_tA\right\|_{p} = \sup_{\|B\|_{p^{\prime}}=1} \left|\tau((A-P_t A)^{\ast} B)\right| \leq \left\||\nabla A|\right\|_p \int_{0}^{t}G^{\text{\rm\ref{prop:Lipschitz-smoothing}}}_{1}(s)^{-1/2}  \dif s
\end{equation*} 
as desired.

Finally, by the ergodicity of $P_t$, i.e., $P_tA\rightarrow \tau(A)\mathbf 1$ as $t\rightarrow\infty$, we obtain
\begin{equation*}
    \left\|A-\tau(A)\right\|_{p}\leq \left\||\nabla A|\right\|_{p} \int_{0}^{\infty}G^{\text{\rm\ref{prop:Lipschitz-smoothing}}}_{1}(s)^{-1/2}  \dif s.
\end{equation*}
Recalling $G^{\text{\rm\ref{prop:Lipschitz-smoothing}}}_1(s)=\frac13(e^s+3)(e^s-1)$, the substitution $u=\sqrt{\frac{e^s-1}{e^s+3}}$ yields
\begin{equation*}
    \int_0^\infty G^{\text{\rm\ref{prop:Lipschitz-smoothing}}}_1(s)^{-1/2}\dif s =2\sqrt3\int_0^1 \frac{\dif u}{1+3u^2} =2\arctan(\sqrt3) = \frac{2\pi}{3},
\end{equation*}
which completes the proof.
\end{proof}

\subsection{Energies, influences and partial variances}

The notion of energy originates in classical analysis. Dating back to Dirichlet's principle and the study of harmonic functions, the Dirichlet energy functional $\Ecal[f]=\left\||\nabla f|\right\|_2^2$ plays a central role in the theory of Sobolev spaces and Dirichlet forms. A natural extension is the $p$-energy, $\Ecal_p[f]=\left\||\nabla f|\right\|_p^p$, which underlies nonlinear potential theory and the variational characterization of solutions to the $p$-Laplace equation.
In particular, when $p=1$, $\Ecal_1[f] = \left\||\nabla f|\right\|_1$ is often interpreted as the total variation, a quantity closely related to boundary measure and crucial in the study of isoperimetric inequalities. In the quantum Boolean setting, for $A\in M_2(\Cbb)^{\otimes n}$ and $p\ge 1$, we analogously define the \emph{$p$-energy} of $A$ by
\begin{equation*}
    \Ecal_p[A]:=\left\||\nabla A|\right\|_p^p,
\end{equation*}
which plays the same role for quantum Boolean cubes, extending the classical notion to the noncommutative framework. Similarly, when $p=1$, $\Ecal_1[A] = \left\||\nabla A|\right\|_1$ is called the \emph{total variation}; when $p=2$, $\Ecal_2[A]$ gives the \emph{Dirichlet energy} $\Ecal[A]$ as in Proposition \ref{prop:quantum-Poincare} and \ref{prop:quantum-hyper}:
\begin{equation*}
    \Ecal_2[A] = \left\||\nabla A|\right\|_2^2 = \sum_{j=1}^n \left\|d_j A\right\|_2^2 = \tup{A,\Lcal A} = \Ecal[A].
\end{equation*}

In parallel, the notion of influence was first introduced for classical Boolean functions to quantify the sensitivity of a function to its coordinates, and has since become a cornerstone of threshold phenomena and concentration of measure. More specifically, for a Boolean function $f:\{-1,1\}^{n}\to \{-1,1\}$, the influence of the $i$-th coordinate on $f$ is defined as the probability that flipping the $i$-th bit changes the value of $f$. This probabilistic definition coincides with the analytic formulation $\Inf_i[f]:=\|D_i f\|_2^2$ for scalar-valued functions, which can be interpreted as a ``partial energy'' localized to the $i$-th coordinate. The concept has since been extended both to other product spaces, such as the standard Gaussian setting, and to different $L^p$-norms, such as the geometric influence $\Inf_i^1[f]:=\|D_i f\|_1$ for $p=1$; see, e.g., \cite{CL2012hypercontr,keller2012geometric,keller2014geometricII}. In the quantum Boolean setting, for $A\in M_2(\Cbb)^{\otimes n}$, $j\in [n]$ and $p\ge 1$, we analogously define the \emph{$p$-influence} of the $j$-th coordinate on $A$, and the \emph{total $p$-influence} on $A$, respectively by
\begin{equation*}
    \Inf^p_j[A]:=\left\|d_j A\right\|_p^p,\quad \Inf^p[A] := \sum_{j=1}^n \Inf^p_j[A],
\end{equation*}
which extend the classical notion to the noncommutative framework. Similarly, when $p=1$, $\Inf^1_j[A] = \left\|d_j A\right\|_1$ is called the \emph{geometric influence}; when $p=2$, the superscript $p$ is usually omitted, and the total influence coincides with the Dirichlet energy:
\begin{equation*}
    \Inf[A] = \sum_{j=1}^n \Inf_j [A] = \sum_{j=1}^n \left\|d_j A\right\|_2^2 = \Ecal[A].
\end{equation*}
Further for general $1\leq p\leq 2$, Lemma~\ref{lem:subadditivity} yields
\begin{equation}
\label{eq:Lp-energy-infuence-comparing}
    \Ecal_p[A] = \tau\left(\Bigl(\sum_{j=1}^n|d_j A|^2\Bigr)^{\frac{p}{2}}\right) \leq \sum_{j=1}^n\tau(|d_jA|^{p})=\Inf^p[A].
\end{equation}

In many problems, first-order information is not sufficient: one needs to quantify the effect of a coalition of variables. This motivates the notion of \emph{high-order influence}, which captures the sensitivity of a function to simultaneous perturbations on a set of coordinates. For nonempty $J\subseteq[n]$, we similarly define the $p$-influence of the $J$-coordinates on $A$ by
\begin{equation*}
    \Inf_J^p[A]:=\left\|d_JA\right\|_p^p.
\end{equation*}

A key observation is that variance can be represented as the integral of total influence along the depolarizing semigroup. Indeed, due to the ergodicity, i.e., $P_t\rightarrow\tau(\cdot) \1$ as $t\rightarrow\infty$, we have
\begin{equation}\label{eq:var-influence-representation}
    \Var(A) = \int_0^\infty2\,\Ecal[P_tA]\dif t = \int_0^\infty2\Inf[P_tA]\dif t.
\end{equation}
Replacing the total influence with high-order influences in \eqref{eq:var-influence-representation} leads to the notion of \emph{partial variances}, which underlie the semigroup-based high-order analysis. For nonempty $J\subseteq[n]$, we define the partial variance restricted to the $J$-coordinates by
\begin{equation}\label{eq:def-local-variance}
    V_J(A):=\int_0^\infty2\Inf_J[P_tA]\dif t=\int_0^\infty2\left\|d_JP_tA\right\|_2^2\,\dif t = \left\|R_J(A)\right\|_2^2.
\end{equation}
where $R_J(A) := d_J \Lcal^{-1/2}(A-\tau(A))$ is the \emph{Riesz transform} with respect to the $J$-coordinates.

Applying Parseval's identity to \eqref{eq:semigroup-Fourier} and \eqref{eq:exp-der-Fourier} yields the Fourier--Pauli formulas:
\begin{equation}\label{eq:high-order-influence-variance-Fourier-formula}
    \Inf_J[A] = \sum_{\substack{s\in \{0,1,2,3\}^n\\\supp(s)\supseteq J}} |\widehat{A}_s|^2,\quad  V_J(A)=\sum_{\substack{s\in \{0,1,2,3\}^n\\\supp(s)\supseteq J}}\frac1{|\supp(s)|}|\widehat{A}_s|^2.
\end{equation}
Consequently, we have the following restricted Poincar\'{e} inequality and Fourier weight bound:
\begin{proposition}\label{prop:restricted-Poincare}
For every $A\in M_2(\Cbb)^{\otimes n}$ and $J\subseteq[n]$ with $|J|=k\geq 1$,
\begin{equation}\label{eq:restricted-Poincare-ineq}
    k V_J(A)\le \Inf_J[A].
\end{equation}
and equivalently, for all $t\geq 0$,
\begin{equation}\label{eq:restricted-Poincare-decay}
    V_J(P_tA)\le e^{-2kt}V_J(A).
\end{equation}
\end{proposition}
\begin{proof}
Comparing the Fourier--Pauli formulas in \eqref{eq:high-order-influence-variance-Fourier-formula} proves \eqref{eq:restricted-Poincare-ineq}. For \eqref{eq:restricted-Poincare-decay}, fix $A\in M_2(\Cbb)^{\otimes n}$ and $J\subseteq[n]$ with $|J|=k\geq 1$. Using  \eqref{eq:restricted-Poincare-ineq} applied to $P_t A$, we obtain
\begin{equation*}
    \frac{\dif}{\dif t} \left(e^{2kt}\, V_J(P_tA)\right) =  2e^{2kt} \left(k\,V_J(P_tA)-\Inf_J[P_tA]\right) \leq 0,
\end{equation*}
i.e., $t\mapsto e^{2kt}\, V_J(P_tA)$ is decreasing, which proves \eqref{eq:restricted-Poincare-decay}.
\end{proof}

\begin{proposition}\label{prop:sumVarJ-Fourier}
For every $A\in M_2(\Cbb)^{\otimes n}$ and $k\in [n]$,
\begin{equation}\label{eq:sumVarJ}
\sum_{\substack{J\subseteq[n]\\ |J|=k}}k\,V_J(A)
\ge W^{\ge k}[A].
\end{equation}
\end{proposition}
\begin{remark}
Inequality \eqref{eq:sumVarJ} is indeed a high-order extension of \eqref{eq:var-Fourier}:
\begin{equation*}
    \sum_{j=1}^n V_j(A) = \int_0^\infty2\sum_{j=1}^n \Inf_j[P_tA]\dif t = \Var(A) = W^{\ge 1}[A].
\end{equation*}
\end{remark}
\begin{proof}
Using the Fourier--Pauli formula~\eqref{eq:high-order-influence-variance-Fourier-formula}, we obtain
\begin{equation*}
    \sum_{\substack{J\subseteq[n]\\ |J|=k}}k\,V_J(A) =\sum_{\substack{J\subseteq[n]\\ |J|=k}}k\,\sum_{\substack{s\in \{0,1,2,3\}^n\\\supp(s)\supseteq J}}\frac{1}{|\supp(s)|}\,|\widehat A_s|^2 = \sum_{m=k}^n \frac{k}{m} \sum_{\substack{J\subseteq[n]\\ |J|=k}} \sum_{\substack{s\in \{0,1,2,3\}^n\\\supp(s)\supseteq J\\ |\supp(s)|=m}} |\widehat A_s|^2
\end{equation*}
Exchanging the order of summation and recalling the definition of Fourier weights~\eqref{eq:Fourier-weights} yield
\begin{equation*}
    \sum_{\substack{J\subseteq[n]\\ |J|=k}} \sum_{\substack{s\in \{0,1,2,3\}^n\\\supp(s)\supseteq J\\ |\supp(s)|=m}} |\widehat A_s|^2 = \sum_{\substack{s\in \{0,1,2,3\}^n\\ |\supp(s)|=m}}\ \sum_{\substack{J\subseteq\supp(s)\\ |J|=k}} |\widehat A_s|^2 = \sum_{\substack{s\in \{0,1,2,3\}^n\\ |\supp(s)|=m}} \binom{m}{k}\, |\widehat A_s|^2 = \binom{m}{k}\,W^{=m}[A].
\end{equation*}
Combining the identities, and using $\frac{k}{m}\binom{m}{k}=\binom{m-1}{k-1}\geq 1$, we obtain
\begin{equation*}
    \sum_{\substack{J\subseteq[n]\\ |J|=k}}k\,V_J(A) = \sum_{m=k}^n \frac{k}{m}\binom{m}{k}\,W^{=m}[A] \geq \sum_{m=k}^n W^{=m}[A] = W^{\ge k}[A]
\end{equation*}
as desired.
\end{proof}

\section{Main Results}\label{sec:proofs-main}

We now present our two principal theorems, which establish dimension-free Talagrand-type inequalities on the quantum Boolean cube with explicit computable constants: the energy--variance inequality Theorem~\ref{thm:energy-variance} and the high-order influence--variance inequality Theorem~\ref{thm:influence-var}. Both results arise from a unified variance-decay perspective:
\begin{enumerate}
    \item\label{item:step(1)} General implication for individual observables: every short-time variance decay along the depolarizing semigroup yields a corresponding Talagrand-type inequality.
    \item\label{item:step(2)} Estimating the rates of variance decay for individual observables via classical functional inequalities such as the Poincar\'{e} inequality and hypercontractivity.
\end{enumerate}
Taken together, these two steps provide a semigroup-based framework of broad applicability, unifying first-order and high-order phenomena under the principle of variance decay. In what follows, Section~\ref{sec:proof-thm11} establishes Theorem~\ref{thm:energy-variance} via Lemma~\ref{lem:quantum-energy-var-ineq} (Step~\ref{item:step(1)}) and Lemma~\ref{lem:quantum-energy-var-decay} (Step~\ref{item:step(2)}), while Section~\ref{sec:proof-thm12} develops Theorem~\ref{thm:influence-var} via Lemma~\ref{lem:quantum-high-order-influence-var-ineq} (Step~\ref{item:step(1)}) and Lemma~\ref{lem:quantum-high-order-influence-var-decay} (Step~\ref{item:step(2)}).

\subsection{Energy--variance inequality}\label{sec:proof-thm11}
In this section, we establish Theorem~\ref{thm:energy-variance}, which provides a dimension-free relation between energies and variances. Here, the inequality involves an amplification functional $\Rscr(A,q)$, which compares the variance against the $L^q$-deviation and the aggregated $q$-influences.
\begin{theorem}\label{thm:energy-variance}
For every $A\in M_2(\Cbb)^{\otimes n}$ and $1\leq q\leq p\leq 2$,
\begin{equation*}
    \left\|A\right\|_{\infty}^{2-p}\, \Ecal_p[A] \geq \Var(A) \, \max\left\{C^{\text{\rm\ref{thm:energy-variance}}}_{1}(p),C^{\text{\rm\ref{thm:energy-variance}}}_{2}(p)\, \Rscr(A,q)^{\frac{p}{2}}\right\}
\end{equation*}
holds with amplification functional
\begin{equation*}
    \Rscr(A,q) = \frac{q}{2(2-q)} \max\left\{\ln \left(\frac{\Var(A)}{\|A-\tau(A)\|_{q}^2}\right),\, \ln^{+}\left(\frac{\Var(A)}{\sum_{j} \Inf^{q}_{j}[A]^{2/q}}\right) \right\},
\end{equation*}
and explicit computable constants
\begin{equation*}
    C^{\text{\rm\ref{thm:energy-variance}}}_{1}(p)=\frac{\left(\frac{3}{4}\right)^{p-1}}{\Beta(\frac{3}{4};\frac{p}{2},\frac{p}{2})},\quad C^{\text{\rm\ref{thm:energy-variance}}}_{2}(p)= \frac{p}{2} \left(\frac{4}{3}\right)^{1-\frac{p}{2}}\max_{x\geq 0} \frac{1-e^{-x}}{x^{\frac{p}{2}}},
\end{equation*}
where $\Beta(x;a,b) := \int_0^x t^{a-1} \left(1-t\right)^{b-1} \dif t$ is the incomplete Beta function.
\end{theorem}
\begin{remark}
The constants satisfy $C^{\text{\rm\ref{thm:energy-variance}}}_{1}(p) \geq C^{\text{\rm\ref{thm:energy-variance}}}_{2}(p) \geq 0.368$, so the inequality is dimension-free and uniformly non-degenerate across $1 \leq p \leq 2$. At $p=1$, one has $C^{\text{\rm\ref{thm:energy-variance}}}_{1}(1)=\frac{3}{2\pi}$, recovering the $L^1$-Poincar\'{e} constant in Proposition \ref{prop:quantum-p-Poincare}. At $p=2$, $C^{\text{\rm\ref{thm:energy-variance}}}_{1}(2) = C^{\text{\rm\ref{thm:energy-variance}}}_{2}(2)=1$, and
\begin{equation*}
    \Var(A)\,\Rscr(A,q) \geq \frac{q}{2(2-q)} \Var(A) \ln \left(\frac{\Var(A)}{\|A-\tau(A)\|_{q}^2}\right) \stackrel{q\rightarrow 2}{\longrightarrow} \frac{1}{2} \,\Ent[|A-\tau(A)|^2],
\end{equation*}
so the inequality recovers both the Poincar\'e inequality \eqref{eq:Poincare-ineq} and the log-Sobolev inequality \eqref{eq:log-Sobolev-ineq}. 
\end{remark}

We now implement Step~\ref{item:step(1)}: for any fixed observable, every short-time variance decay along the depolarizing semigroup yields a corresponding energy--variance inequality. This general implication reduces the task of establishing Talagrand-type energy--variance inequalities to that of finding suitable short-time variance decay estimates. The following lemma makes this precise.

\begin{lemma}
\label{lem:quantum-energy-var-ineq}
Let $A\in M_2(\Cbb)^{\otimes n}$. If for some $\Rscr,\epsilon>0$, we have the short-time variance decay:
\begin{equation}\label{eq:var-decay-condition}
    \Var(P_{t} A) \leq e^{-2\Rscr t} \Var(A), \quad \forall\, t\in [0,\epsilon],
\end{equation}
then for all $p\in[1,2]$, the following energy--variance inequality holds:
\begin{equation}\label{ineq:main quantum improved Talagrand}
    \left\|A\right\|_{\infty}^{2-p}\, \Ecal_p[A] \geq  \left(\frac{3}{4}\right)^{p-1} \Var(A)\, \sup_{t\in (0,\epsilon]} \frac{1-e^{-2\Rscr t}}{ \Beta\left(\frac{3}{4}\left(1-e^{-2t}\right);\frac{p}{2},\frac{p}{2}\right)}.
\end{equation}
\end{lemma}

The proof strategy is to bound $\Inf[P_sA]$ in terms of the $p$-energy $\Ecal_p[A]$ and the $L^\infty$-norm $\|A\|_\infty$, using the derivative decay (Lemma~\ref{lem:j-derivative-decay}) and the Lipschitz smoothing (Proposition~\ref{prop:Lipschitz-smoothing}) estimates. Together with the integral representation of variance \eqref{eq:var-influence-representation} in terms of $\Inf[P_sA]$, the assumed short-time variance decay then yields directly the desired energy--variance relation.

\begin{proof}
Fix $A\in M_2(\Cbb)^{\otimes n}$. We begin by establishing the bound that for all $s> 0$ and $p\in [1,2]$,
\begin{equation}\label{eq:influence-bound}
    \Inf[P_s A] \leq e^{-2(p-1) s} G^{\text{\rm\ref{prop:Lipschitz-smoothing}}}_{1}(2s)^{-(1-\frac{p}{2})}  \left\|A\right\|_{\infty}^{2-p}  \Ecal_p[A].
\end{equation}
Note that $P_s$ is tracially symmetric and commutes with $d_j$'s. Hence
\begin{align*}
    \Inf[P_sA] &=\sum_{j=1}^n \left\|d_jP_sA\right\|_2^2 =\sum_{j=1}^n \tup{d_jA, d_j P_{2s} A} =\tau\Bigl(\sum_{j=1}^n (d_jA)^\ast (d_jP_{2s}A)\Bigr),
\end{align*}
Applying the noncommutative vector-valued H\"older inequality (Lemma~\ref{lem:nc-holder}) yields
\begin{equation*}
    \left|\tau\Bigl(\sum_{j=1}^n(d_jA)^*(d_jP_{2s}A)\Bigr)\right|\le \left\|\Bigl(\sum_{j=1}^n |d_j A|^2\Bigr)^{1/2}\right\|_p \left\|\Bigl(\sum_{j=1}^n |d_jP_{2s} A|^2\Bigr)^{1/2}\right\|_{p'} = \left\||\nabla A|\right\|_{p} \left\||\nabla P_{2s} A|\right\|_{p^{\prime}}.
\end{equation*}
where $p^{\prime}=\frac{p}{p-1}\in[2,\infty]$ is the H\"{o}lder conjugate of $p$. By the log-convexity of the normalized Schatten norm (Lemma~\ref{lem:holder-interp-Schatten}),
\begin{align*}
    &\left\||\nabla P_{2s} A|\right\|_{p^{\prime}} \leq \left\||\nabla P_{2s} A|\right\|_{2}^{\frac{2}{p^{\prime}}} \left\||\nabla P_{2s} A|\right\|_{\infty}^{\frac{p^{\prime}-2}{p^{\prime}}}= \left\||\nabla P_{2s} A|\right\|_{2}^{2\frac{p-1}{p}} \left\||\nabla P_{2s} A|\right\|_{\infty}^{\frac{2-p}{p}}.
\end{align*}
\begin{itemize}
    \item For the $L^2$-term, Lemma \ref{lem:j-derivative-decay} applied to $P_s A$ yields
    \begin{equation*}
        \||\nabla P_{2s}A|\|_2^2 = \sum_{j=1}^n\|d_j P_s (P_s A) \|_2^2 \le e^{-2s}\sum_{j=1}^n\|d_j(P_sA)\|_2^2 = e^{-2s}\Inf[P_sA].
    \end{equation*}
    \item For the $L^\infty$-term, Proposition~\ref{prop:Lipschitz-smoothing} gives
    \begin{equation*}
        \left\||\nabla P_{2s} A|\right\|_{\infty} \leq G^{\text{\rm\ref{prop:Lipschitz-smoothing}}}_{1}(2s)^{-1/2} \|A\|_{\infty}.
    \end{equation*}
\end{itemize}
Combining these bounds gives
\begin{equation*}
    \Inf[P_s A] \leq \left\||\nabla A|\right\|_{p} \left(e^{-2s} \Inf[P_s A]\right)^{\frac{p-1}{p}} \left(G^{\text{\rm\ref{prop:Lipschitz-smoothing}}}_{1}(2s)^{-1/2}\|A\|_{\infty}\right)^{\frac{2-p}{p}}.
\end{equation*}
If $\Inf[P_sA]=0$ the bound is trivial. Otherwise, recalling $\Ecal_p[A]=\|\nabla A\|_p^p$, raising both sides to the power $p$ and canceling $\Inf[P_sA]^{p-1}$ yields the desired bound
\begin{equation*}
    \Inf[P_s A] \leq e^{-2(p-1) s} G^{\text{\rm\ref{prop:Lipschitz-smoothing}}}_{1}(2s)^{-(1-\frac{p}{2})}  \left\|A\right\|_{\infty}^{2-p}  \Ecal_p[A].
\end{equation*}
Now, for every $t\in(0,\varepsilon]$, combining together the variance-decay assumption \eqref{eq:var-decay-condition}, the variance representation  \eqref{eq:var-influence-representation}, and the bound \eqref{eq:influence-bound}, we obtain
\begin{align*}
    \left(1-e^{-2\Rscr t}\right)\Var(A) &\leq \Var(A) - \Var(P_{t} A) =\int_{0}^{t} 2\Inf[P_s A] \dif s \\
    &\leq  \left\|A\right\|_{\infty}^{2-p} \Ecal_p[A] \int_{0}^{t} 2 e^{-2(p-1) s} G^{\text{\rm\ref{prop:Lipschitz-smoothing}}}_{1}(2s)^{-(1-\frac{p}{2})} \dif s.
\end{align*}
Recalling $G^{\text{\rm\ref{prop:Lipschitz-smoothing}}}_1(2s)=\frac13(e^{2s}+3)(e^{2s}-1)$, the substitution $r=\frac{3}{4}(1-e^{-2s})$ gives
\begin{align*}
    &\int_{0}^{t} 2e^{-2(p-1) s} G^{\text{\rm\ref{prop:Lipschitz-smoothing}}}_{1}(2s)^{-(1-\frac{p}{2})} \dif s = \int_{0}^{t} 2e^{-2(p-1) s} \left[\frac{1}{3}(e^{2s}+3)(e^{2s}-1)\right]^{-(1-\frac{p}{2})} \dif s \\
    &\qquad= \left(\frac{4}{3}\right)^{p-1} \int_{0}^{\frac{3}{4}(1-e^{-2t})} r^{\frac{p}{2}-1} \left(1-r\right)^{\frac{p}{2}-1} \dif r =\left(\frac{4}{3}\right)^{p-1} \Beta\left(\frac{3}{4}\left(1-e^{-2t}\right);\frac{p}{2},\frac{p}{2}\right).
\end{align*}
Taking the supremum over $t\in(0,\varepsilon]$ yields the claimed inequality.
\end{proof}

Next, we proceed to Step \ref{item:step(2)}: for individual observables, we derive suitable estimates of the rate $\Rscr$ governing the $\varepsilon$-short-time variance decay \eqref{eq:var-decay-condition}.
\begin{lemma}
\label{lem:quantum-energy-var-decay}
Let $A\in M_2(\Cbb)^{\otimes n}$. Then the short-time variance decay \eqref{eq:var-decay-condition} holds with
\begin{itemize}
    \item $\Rscr=1$ for every $\varepsilon>0$;
    \item $\Rscr = \frac{\tanh(\varepsilon)}{\varepsilon}\Rscr(A,q)$ for every $0<\varepsilon \leq \arctanh(\frac{2-q}{q})$ and $q\in [1,2)$.
\end{itemize}
\end{lemma}

The first statement follows immediately from the Poincar\'e inequality \eqref{eq:Poincare-var-decay}. For the second statement, the key step is to establish the endpoint estimate via hypercontractivity \eqref{eq:hypercontractivity} applied to the variance representation
\begin{equation*}
    \Var(P_\varepsilon A) = \left\|P_\varepsilon (A - \tau(A))\right\|_2^2 = \sum_{j=1}^{n} \left\|P_{\varepsilon} \tau_1 \tau_2\cdots \tau_{j-1} (d_j A)\right\|_2^2.
\end{equation*}
Together with the log-convexity of $t\mapsto \Var(P_t A)$, the endpoint estimate extends to the desired short-time variance decay estimate \eqref{eq:var-decay-condition}. The details are given below.


\begin{proof}
Let $A\in M_2(\Cbb)^{\otimes n}$. The first statement is a direct consequence of the Poincar\'{e} inequality \eqref{eq:Poincare-var-decay}, which gives the uniform decay rate $\Rscr=1$ for every $\varepsilon>0$. For the second statement, let $q\in [1,2)$ and $0<\varepsilon \leq \arctanh(\frac{2-q}{q})$. It suffices to verify the endpoint estimate
\begin{equation}\label{ineq:endpoint-var-decay}
    \Var(P_\varepsilon A) \leq e^{-2\Rscr \varepsilon}\Var(A).
\end{equation}
Note that $t\mapsto \Var(P_t A)$ is log-convex as a mixture of exponential functions:
\begin{equation*}
    \Var(P_t A) = \sum_{\substack{s\in\{0,1,2,3\}^n\\ s\neq 0}} e^{-2t|\supp(s)|} |\widehat{A}_s|^2.
\end{equation*}
If \eqref{ineq:endpoint-var-decay} holds, then the log-convexity of $t\mapsto \Var(P_t A)$ yields that for every $t\in [0,\varepsilon]$,
\begin{equation*}
    \Var(P_t A) \leq \Var(P_0 A)^{1-\frac{t}{\varepsilon}}\Var(P_{\varepsilon}A)^{\frac{t}{\varepsilon}} = \Var(A)^{1-\frac{t}{\varepsilon}}\Var(P_{\varepsilon}A)^{\frac{t}{\varepsilon}}\le e^{-2\Rscr t}\Var(A),
\end{equation*}
which establishes the desired short-time variance decay.

Recall that $\Rscr = \frac{\tanh(\varepsilon)}{\varepsilon}\Rscr(A,q)$, where
\begin{equation*}
    \Rscr(A,q) = \frac{q}{2(2-q)} \max\left\{\ln \left(\frac{\Var(A)}{\|A-\tau(A)\|_{q}^2}\right),\, \ln^{+}\left(\frac{\Var(A)}{\sum_{j} \Inf^{q}_{j}[A]^{2/q}}\right) \right\},
\end{equation*}
Since $\left\|A-\tau(A) \right\|_{q}^2\leq \Var(A)$, it follows that
\begin{align*}
    e^{-2\Rscr \varepsilon} &= \left(\min\left\{\frac{\left\|A-\tau(A) \right\|_{q}^2}{\Var(A)},\ \min\left\{\frac{\sum_{j=1}^{n}  \Inf^{q}_{j}[A]^{2/q}}{\Var(A)},\ 1 \right\}\right\}\right)^{\vartheta}\\
    &= \Var(A)^{-\vartheta} \Bigg(\min\bigg\{\left\|A-\tau(A) \right\|_{q}^2,\ \sum_{j=1}^{n}  \left\| d_j A\right\|_q^2\bigg\}\Bigg)^{\vartheta},
\end{align*}
where $\vartheta:=\frac{q}{2-q}\,\tanh(\varepsilon) \in (0,1]$. Hence, verifying \eqref{ineq:endpoint-var-decay} reduces to showing
\begin{equation}\label{ineq:quantum-variance-bound-rewrite}
    \Var(P_\varepsilon A) \leq \Var(A)^{1-\vartheta} \Bigg(\min\bigg\{\left\|A-\tau(A) \right\|_{q}^2,\ \sum_{j=1}^{n}  \left\| d_j A\right\|_q^2\bigg\}\Bigg)^{\vartheta}.
\end{equation}
\begin{itemize}
    \item First, we establish the bound
    \begin{equation*}
        \Var(P_\varepsilon A) \leq \Var(A)^{1-\vartheta} \left\|A-\tau(A) \right\|_{q}^{2\vartheta}.
    \end{equation*}
    Since $P_\varepsilon$ commutes with $\tau$,
    \begin{equation*}
        \Var(P_\varepsilon A) = \left\|P_\varepsilon A - \tau(P_\varepsilon A)\right\|_2^2 = \left\|P_\varepsilon (A - \tau(A))\right\|_2^2.
    \end{equation*}
    By hypercontractivity \eqref{eq:hypercontractivity},
    \begin{equation*}
        \left\|P_{\varepsilon} (A - \tau(A)) \right\|_{2} \leq \left\|A-\tau(A) \right\|_{1+e^{-2\varepsilon}}.
    \end{equation*}
    Since $\frac{1}{1+e^{-2\varepsilon}}=\frac{1-\vartheta}{2}+\frac{\vartheta}{q}$, log-convexity of the norm (Lemma~\ref{lem:holder-interp-Schatten}) yields
    \begin{align*}
        \left\|A-\tau(A) \right\|_{1+e^{-2\varepsilon}} \leq \left\|A-\tau(A) \right\|_{2}^{1-\vartheta} \left\|A-\tau(A) \right\|_{q}^{\vartheta}.
    \end{align*}
    Rounding up, we obtain
    \begin{equation*}
        \Var(P_\varepsilon A) \leq \left(\left\|A-\tau(A) \right\|_{2}^{1-\vartheta} \left\|A-\tau(A) \right\|_{q}^{\vartheta}\right)^{2} = \Var(A)^{1-\vartheta} \left\|A-\tau(A) \right\|_{q}^{2\vartheta}
    \end{equation*}
    as desired.

    \item Next, we establish the bound
    \begin{equation*}
        \Var(P_\varepsilon A) \leq  \Var(A)^{1-\vartheta} \bigg(\sum_{j=1}^{n}  \left\| d_j A\right\|_q^2\bigg)^{\vartheta}.
    \end{equation*}
    Denote $\tau_{<j} = \tau_1\tau_2\cdots \tau_{j-1}$. The standard conditional variance decomposition gives
    \begin{equation}\label{eq:var-martingale-decomposition}
        \Var(A) = \|A\|_2^2 - |\tau(A)|^2 = \sum_{j=1}^{n} \left(\left\|\tau_{<j} (A)\right\|_2^2- \left\|\tau_{<j}\tau_j (A)\right\|_2^2\right) = \sum_{j=1}^{n} \left\|\tau_{<j} (d_j A)\right\|_2^2.
    \end{equation}
    Since $P_\varepsilon$ commutes with $\tau_{<j}$ and $d_j$,
    \begin{equation*}
        \Var(P_{\varepsilon} A) = \sum_{j=1}^{n} \left\|\tau_{<j} (d_j P_{\varepsilon} A)\right\|_2^2 = \sum_{j=1}^{n} \left\|P_{\varepsilon} \tau_{<j} (d_j A)\right\|_2^2.
    \end{equation*}
    Applying hypercontractivity and norm log-convexity as above,
    \begin{equation*}
        \left\|P_{\varepsilon} \tau_{<j} (d_j A)\right\|_2 \leq \left\|\tau_{<j} (d_j A)\right\|_{1+e^{-2\varepsilon}}\le \left\|\tau_{<j} (d_j A)\right\|_{2}^{1-\vartheta} \left\|\tau_{<j} (d_j A)\right\|_{q}^{\vartheta}.
    \end{equation*}
    By the H\"older inequality in the sum over $j$,
    \begin{equation*}
        \sum_{j=1}^{n} \left(\left\|\tau_{<j} (d_j A)\right\|_{2}^{1-\vartheta} \left\|\tau_{<j} (d_j A)\right\|_{q}^{\vartheta}\right)^2 \leq \bigg(\sum_{j=1}^{n} \left\|\tau_{<j} (d_j A)\right\|_{2}^2 \bigg)^{1-\vartheta} \bigg(\sum_{j=1}^{n} \left\|\tau_{<j} (d_j A)\right\|_{q}^2 \bigg)^{\vartheta}.
    \end{equation*}
    The first term reproduces $\Var(A)$ by \eqref{eq:var-martingale-decomposition}, while the $L^q$-contractivity of $\tau_{<j}$ controls the second term via $\|\tau_{<j}(d_jA)\|_q\leq\|d_jA\|_q$. Rounding up, we obtain
    \begin{equation*}
        \Var(P_{\varepsilon} A) \leq \Var(A)^{1-\vartheta}\bigg(\sum_{j=1}^{n} \left\|d_j A\right\|_{q}^2 \bigg)^{\vartheta}
    \end{equation*}
    as desired.
\end{itemize}
Taking the stronger of the two bounds yields \eqref{ineq:quantum-variance-bound-rewrite}, which completes the proof.
\end{proof}

Finally, combining the general implication (Lemma~\ref{lem:quantum-energy-var-ineq}) with the short-time variance decay estimates (Lemma~\ref{lem:quantum-energy-var-decay}), we complete the proof of Theorem~\ref{thm:energy-variance}:
\begin{proof}[Proof of Theorem~\ref{thm:energy-variance}]
Let $A\in M_2(\Cbb)^{\otimes n}$ and $1\leq q\leq p\leq 2$.
\begin{itemize}
    \item First, we prove the bound
    \begin{equation*}
        \left\|A\right\|_{\infty}^{2-p}\, \Ecal_p[A] \geq C^{\text{\rm\ref{thm:energy-variance}}}_{1}(p)\, \Var(A).
    \end{equation*}
    Taking $\Rscr=1$ with $\varepsilon=\infty$ in \eqref{ineq:main quantum improved Talagrand} yields
    \begin{align*}
        \left\|A\right\|_{\infty}^{2-p}\, \Ecal_p[A] &\geq \left(\frac{3}{4}\right)^{p-1} \Var(A)\, \sup_{t>0} \frac{1-e^{-2 t}}{ \Beta\left(\frac{3}{4}\left(1-e^{-2t}\right);\frac{p}{2},\frac{p}{2}\right)}\\
        &\geq \left(\frac{3}{4}\right)^{p-1} \Var(A)\, \lim_{t\rightarrow\infty} \frac{1-e^{-2 t}}{ \Beta\left(\frac{3}{4}\left(1-e^{-2t}\right);\frac{p}{2},\frac{p}{2}\right)} = C^{\text{\rm\ref{thm:energy-variance}}}_{1}(p)\, \Var(A)
    \end{align*}
    with $C^{\text{\rm\ref{thm:energy-variance}}}_{1}(p)=\left(\frac{3}{4}\right)^{p-1}/\Beta\left(\frac{3}{4};\frac{p}{2},\frac{p}{2}\right)$ as desired.
    \item Next, we prove the bound
    \begin{equation*}
        \left\|A\right\|_{\infty}^{2-p}\, \Ecal_p[A] \geq C^{\text{\rm\ref{thm:energy-variance}}}_{2}(p)\, \Var(A)\, \Rscr(A,q)^{\frac{p}{2}}.
    \end{equation*}
    Optimizing \eqref{ineq:main quantum improved Talagrand} with $\Rscr = \max\{1,\frac{\tanh(\varepsilon)}{\varepsilon}\Rscr(A,q)\}$ over $0<\varepsilon \leq \arctanh(\frac{2-q}{q})$ gives
    \begin{equation*}
        \left\|A\right\|_{\infty}^{2-p}\, \Ecal_p[A] \geq  \left(\frac{3}{4}\right)^{p-1} \Var(A)\, \sup_{0<\varepsilon \leq \arctanh(\frac{2-q}{q})}\sup_{t\in (0,\epsilon]} \frac{1-e^{-2t \max\{1,\frac{\tanh(\varepsilon)}{\varepsilon}\Rscr(A,q)\}}}{ \Beta\left(\frac{3}{4}\left(1-e^{-2t}\right);\frac{p}{2},\frac{p}{2}\right)}.
    \end{equation*}
    Since $\max\{1,\frac{\tanh(\varepsilon)}{\varepsilon}\Rscr(A,q)\} \geq \frac{\tanh(\varepsilon)}{\varepsilon} \max\{1,\Rscr(A,q)\}$, we have
    \begin{equation*}
        \sup_{t\in (0,\epsilon]} \frac{1-e^{-2t \max\{1,\frac{\tanh(\varepsilon)}{\varepsilon}\Rscr(A,q)\}}}{ \Beta\left(\frac{3}{4}\left(1-e^{-2t}\right);\frac{p}{2},\frac{p}{2}\right)} \geq \frac{1-e^{-2\tanh(\varepsilon) \max\{1,\Rscr(A,q)\}}}{ \Beta\left(\frac{3}{4}\left(1-e^{-2\varepsilon}\right);\frac{p}{2},\frac{p}{2}\right)}.
    \end{equation*}
    Substituting $r=\tanh(\varepsilon)$ (equivalently, $e^{-2\varepsilon} = \frac{1-r}{1+r}$) and applying Lemma~\ref{lem:beta_upper}, we have
    \begin{equation*}
        \frac{1-e^{-2\tanh(\varepsilon) \max\{1,\Rscr(A,q)\}}}{ \Beta\left(\frac{3}{4}\left(1-e^{-2\varepsilon}\right);\frac{p}{2},\frac{p}{2}\right)} = \frac{1-e^{-2r\max\{1,\Rscr(A,q)\}}}{\Beta\left(\frac{3}{4}\frac{2r}{1+r};\frac{p}{2},\frac{p}{2}\right)} \geq \frac{p}{2} \left(\frac{2}{3}\right)^{\frac{p}{2}} \frac{1-e^{-2r \max\{1,\Rscr(A,q)\}}}{r^{\frac{p}{2}}}
    \end{equation*}
    Setting $\tilde{r}=2\max\{1,\Rscr(A,q)\}\cdot r$ yields
    \begin{equation*}
        \frac{1-e^{-2r \max\{1,\Rscr(A,q)\}}}{r^{\frac{p}{2}}} = \frac{1-e^{-\tilde{r}}}{\tilde{r}^{\frac{p}{2}}} \left(2\max\{1,\Rscr(A,q)\}\right)^{\frac{p}{2}} \geq 2^{\frac{p}{2}} \frac{1-e^{-\tilde{r}}}{\tilde{r}^{\frac{p}{2}}} \Rscr(A,q)^{\frac{p}{2}}.
    \end{equation*}
    Note that $\tilde{r}$ varies over $0<\tilde{r}\leq \frac{2(2-q)}{q} \max\{1,\Rscr(A,q)\}$, and for $q\leq p$ we have
    \begin{equation*}
        \frac{2(2-q)}{q} \max\{1,\Rscr(A,q)\} \geq \frac{2(2-q)}{q} \geq \frac{2(2-p)}{p}.
    \end{equation*}
    Lemma~\ref{lem:phi_max_location} ensures that optimizing $\tfrac{1-e^{-\tilde{r}}}{\tilde{r}^{p/2}}$ over this admissible range yields $\max_{x\geq 0}\tfrac{1-e^{-x}}{x^{p/2}}$. Combining the estimates, we obtain
    \begin{equation*}
        \left\|A\right\|_{\infty}^{2-p}\, \Ecal_p[A] \geq  C^{\text{\rm\ref{thm:energy-variance}}}_{2}(p)\, \Var(A)\, \Rscr(A,q)^{\frac{p}{2}}
    \end{equation*}
    with $C^{\text{\rm\ref{thm:energy-variance}}}_{2}(p) = \frac{p}{2} \left(\frac{4}{3}\right)^{1-\frac{p}{2}}\max_{x\geq 0} \frac{1-e^{-x}}{x^{p/2}}$ as desired.
\end{itemize}
Comparing the two bounds and taking the stronger one completes the proof.
\end{proof}

\subsection{High-order influence--variance inequality}\label{sec:proof-thm12}



In this section, we establish Theorem~\ref{thm:influence-var}, which provides a dimension-free relation between high-order influences and partial variances (the bound depends only on the order $k=|J|$, not on the ambient dimension $n$).

\begin{theorem}
\label{thm:influence-var}
For every $A\in M_2(\Cbb)^{\otimes n}$, $J\subseteq [n]$ with $|J|=k\ge 1$, $p\in [1,2]$, and $q\in [1,2)$,
\begin{equation*}
    \left\|A\right\|_{\infty}^{2-p}\, \Inf_J^p[A] \geq  3^{-k(1-\frac{p}{2})} \, V_J(A)\, \max\left\{k,\,\frac{q}{2(2-q)} \ln^{+} \left(\frac{k\, V_J(A)}{\Inf_J^q[A]^{2/q}}\right)\right\}.
\end{equation*}
\end{theorem}

The proof follows the same two-step strategy as in Theorem~\ref{thm:energy-variance}, now with energies replaced by influences and global variances replaced by partial variances.

We now implement Step~\ref{item:step(1)}: for any fixed observable and subset of coordinates, every short-time decay of the partial variance along the depolarizing semigroup leads to a corresponding influence--variance inequality. This general implication reduces the task of proving Talagrand-type bounds in the high-order setting to establishing suitable short-time decay estimates for partial variances. The following lemma makes this precise.

\begin{lemma}
\label{lem:quantum-high-order-influence-var-ineq}
Let $A\in M_2(\Cbb)^{\otimes n}$ and $J\subseteq[n]$ with $|J|=k\ge 1$. If for some $\Rscr,\epsilon>0$, we have the short-time decay of the $J$-partial variance:
\begin{equation}\label{ineq:quantum-high-order-variance-decay}
    V_J(P_t A) \leq e^{-2\Rscr t} V_J(A), \quad \forall\, t\in [0,\epsilon],
\end{equation}
then for all $p\in [1,2]$, the following influence--variance inequality holds:
\begin{equation}\label{ineq:quantum-high-order-influence-var}
    \left\|A\right\|_{\infty}^{2-p}\, \Inf_J^p[A] \geq 3^{-k(1-\frac{p}{2})}\, V_J(A)\,  \Rscr.
\end{equation}
\end{lemma}

The proof strategy parallels that of Lemma~\ref{lem:quantum-energy-var-ineq}, but with the roles of energies and variances replaced by influences and partial variances. Specifically, we bound $\Inf_J[P_sA]$ in terms of the $p$-influence $\Inf_J^p[A]$ and the $L^\infty$-norm $\|A\|_\infty$, using the derivative decay (Lemma~\ref{lem:j-derivative-decay}) and the operator norm (Proposition~\ref{prop:derivative-infty-norm}) estimates. Together with definition of the $J$-partial variance \eqref{eq:def-local-variance} in terms of $\Inf_J[P_sA]$, the assumed short-time partial variance decay then yields directly the desired high-order influence--variance inequality.

\begin{proof}
Fix $A\in M_2(\Cbb)^{\otimes n}$ and $J\subseteq[n]$ with $|J|=k\ge 1$. We begin by establishing the bound that for all $s>0$ and $p\in [1,2]$,
\begin{equation}\label{eq:J-influence-bound}
    \Inf_J[P_s A] \leq 3^{k(1-\frac{p}{2})} e^{-2ks} \left\|A\right\|_{\infty}^{2-p}  \Inf_J^p[A].
\end{equation}
Note that $P_s$ is tracially symmetric and commutes with $d_J$. Hence
\begin{align*}
    \Inf_J[P_s A] =\left\|d_J P_s A\right\|_2^2 = \tup{d_J A, d_J P_{2s} A}.
\end{align*}
Applying the noncommutative H\"older inequality (Lemma~\ref{lem:nc-holder}) yields
\begin{equation*}
    \left|\tup{d_J A, d_J P_{2s} A}\right| \leq \left\|d_J A\right\|_p \left\|d_J P_{2s} A\right\|_{p^\prime},
\end{equation*}
where $p^{\prime}=\frac{p}{p-1}\in[2,\infty]$ is the H\"{o}lder conjugate of $p$. By the log-convexity of the normalized Schatten norm  (Lemma~\ref{lem:holder-interp-Schatten}),
\begin{equation*}
    \left\|d_J P_{2s} A\right\|_{p^{\prime}} \leq \left\|d_J P_{2s} A\right\|_{2}^{\frac{2}{p^{\prime}}} \left\|d_J P_{2s} A\right\|_{\infty}^{\frac{p^{\prime}-2}{p^{\prime}}}= \left\|d_J P_{2s} A\right\|_{2}^{2\frac{p-1}{p}} \left\|d_J P_{2s} A\right\|_{\infty}^{\frac{2-p}{p}}.
\end{equation*}
\begin{itemize}
    \item For the $L^2$-term, Lemma \ref{lem:j-derivative-decay} applied to $P_s A$ yields
    \begin{equation*}
        \|d_J P_{2s} A\|_2^2 = \|d_J P_s (P_s A) \|_2^2 \le e^{-2ks}\|d_J(P_sA)\|_2^2 = e^{-2ks}\Inf_J[P_sA].
    \end{equation*}
    \item For the $L^\infty$-term, combining Lemma~\ref{lem:j-derivative-decay} and Proposition \ref{prop:derivative-infty-norm} gives
    \begin{equation*}
        \left\|d_J P_{2s} A\right\|_{\infty} \leq e^{-2ks} \left\|d_J A\right\|_{\infty} \leq \sqrt{3}^k e^{-2ks}  \left\|A\right\|_{\infty}.
    \end{equation*}
\end{itemize}
Combining these bounds gives
\begin{equation*}
    \Inf_J[P_s A] \leq \left\|d_J A\right\|_p \left(e^{-2ks} \Inf_J[P_s A]\right)^{\frac{p-1}{p}} \left(\sqrt{3}^k e^{-2ks}\|A\|_{\infty}\right)^{\frac{2-p}{p}}.
\end{equation*}
If $\Inf_J[P_sA]=0$ the bound is trivial. Otherwise, recalling $\Inf_J^p[A]=\left\|d_J A\right\|_{p}^{p}$, raising both sides to the power $p$ and canceling $\Inf_J[P_sA]^{p-1}$ yields the desired bound
\begin{align*}
    \Inf_J[P_s A] \leq 3^{k(1-\frac{p}{2})} e^{-2ks} \left\|A\right\|_{\infty}^{2-p}  \Inf_J^p[A].
\end{align*}
Now, for every $t\in(0,\varepsilon]$, combining together the partial variance decay assumption \eqref{ineq:quantum-high-order-variance-decay}, the definition of partial variance \eqref{eq:def-local-variance}, and the bound \eqref{eq:J-influence-bound}, we obtain
\begin{align*}
    \left(1-e^{-2\Rscr t}\right)V_J(A) &\leq V_J(A)- V_J(P_t A) = \int_{0}^{t} 2\Inf_J[P_s A]  \dif s\\
    &\leq 3^{k(1-\frac{p}{2})}\,\frac{1-e^{-2 k t}}{k}\,\left\|A\right\|_{\infty}^{2-p}  \Inf_J^p[A].
\end{align*}
Taking supremum over $t\in(0,\varepsilon]$ yields
\begin{equation*}
    \left\|A\right\|_{\infty}^{2-p}  \Inf_J^p[A] \geq 3^{-k(1-\frac{p}{2})}\, V_J(A)\,  \sup_{t\in (0,\epsilon]} k\cdot \frac{1-e^{-2\Rscr t}}{1-e^{-2k t}} = 3^{-k(1-\frac{p}{2})}\, V_J(A)\,\Rscr,
\end{equation*}
which completes the proof.
\end{proof}
\begin{remark}
Lemma~\ref{lem:quantum-high-order-influence-var-ineq} also admits a simpler argument. Differentiating \eqref{ineq:quantum-high-order-variance-decay} at $t=0$ yields
\begin{equation*}
    \Inf_J[A] \geq V_J[A]\, \Rscr.
\end{equation*}
On the other hand, the operator norm estimate (Proposition \ref{prop:derivative-infty-norm}) gives
\begin{equation*}
    \Inf_J[A] = \left\|d_J A\right\|_2^2 \leq \left\|d_J A\right\|_\infty^{2-p} \left\|d_J A\right\|_p^p \leq 3^{k(1-\frac{p}{2})}  \left\|A\right\|_\infty^{2-p} \Inf_J^p[A].
\end{equation*}
Combining the two bounds completes the proof. When $p=2$, the dependence on the order $k$ is eliminated; when $p=1$, it can also be eliminated via the improved estimate
\begin{equation*}
    \Inf_J[A] = \left\|d_J A\right\|_2^2 = \tup{A, d_J A} \leq \left\|A\right\|_\infty  \left\|d_J A\right\|_1 = \left\|A\right\|_\infty \Inf_J^1[A].
\end{equation*}
For intermediate values of $p$, while the order-dependence cannot be eliminated for $d_J$ to our best knowledge, considering other choices of high-order derivative, such as the conditional averaged gradient $|\nabla^{0}_J A|:=\tau_J(|d_J A|^2)^{1/2}$, may yield order-independent estimates.
%
In the classical case, the bound $\left\|D_J\right\|_{\infty\to \infty}\leq 1$ immediately yields order-independent estimates.
\end{remark}

Next, we proceed to Step~\ref{item:step(2)}: for individual observables and subsets of coordinates, we derive suitable estimates of the rate $\Rscr$ governing the $\varepsilon$-short-time partial variance decay \eqref{ineq:quantum-high-order-variance-decay}.
\begin{lemma}
\label{lem:quantum-high-order-influence-var-decay}
Let $A\in M_2(\Cbb)^{\otimes n}$ and $J\subseteq[n]$ with $|J|=k\ge 1$. Then the short-time decay of the $J$-partial variance  \eqref{ineq:quantum-high-order-variance-decay} holds with
\begin{itemize}
    \item $\Rscr=k$ for every $\varepsilon>0$;
    \item $\Rscr = \frac{\tanh(\varepsilon)}{\varepsilon} \frac{q}{2(2-q)}\ln^+\left(\frac{k V_J(A)}{\Inf_J^q[A]^{2/q}}\right)$ for every $0<\varepsilon \leq \arctanh(\frac{2-q}{q})$ and $q\in [1,2)$.
\end{itemize}
\end{lemma}


The proof strategy mirrors that of Lemma~\ref{lem:quantum-energy-var-decay}. The first statement follows immediately from the restricted Poincar\'e inequality \eqref{eq:restricted-Poincare-decay}. For the second statement, since $t\mapsto V_J(P_t A)$ is also log-convex, a similar endpoint estimate via hypercontractivity yields the desired short-time partial variance decay  \eqref{ineq:quantum-high-order-variance-decay} with the stated rate. The details are given below.

\begin{proof}
Let $A\in M_2(\Cbb)^{\otimes n}$ and $J\subseteq[n]$ with $|J|=k\ge 1$. The first statement is an immediate consequence of the restricted Poincar\'e inequality \eqref{eq:restricted-Poincare-decay}, which yields the uniform decay rate $\Rscr=k$ for every $\varepsilon>0$. For the second statement, let $q\in [1,2)$ and $0<\varepsilon \leq \arctanh(\frac{2-q}{q})$. Note that $t\mapsto V_J(P_t A)$ is also log-convex as a mixing of exponential functions:
\begin{equation*}
    V_J(P_t A) = \sum_{\substack{s\in\{0,1,2,3\}^n\\ \supp(s) \supseteq J}} e^{-2t|\supp(s)|}\, \frac1{|\supp(s)|}\, |\widehat{A}_s|^2.
\end{equation*}
It suffices to verify the endpoint estimate
\begin{equation}\label{ineq:endpoint-J-var-decay}
    V_J(P_\varepsilon A) \leq e^{-2\Rscr \varepsilon}V_J(A).
\end{equation}
With $\Rscr = \frac{\tanh(\varepsilon)}{\varepsilon} \frac{q}{2(2-q)}\ln^+\left(\frac{k V_J(A)}{\Inf_J^q[A]^{2/q}}\right)$ as defined,
we have
\begin{align*}
    e^{-2\Rscr \varepsilon} &= \left(\min\left\{\frac{\Inf^{q}_{J}[A]^{2/q}}{k\, V_J(A)},\ 1 \right\}\right)^{\vartheta} = V_J(A)^{-\vartheta} \left(\min\left\{k^{-1}\left\| d_J A\right\|_q^2,\ V_J(A)\right\}\right)^{\vartheta},
\end{align*}
where $\vartheta:=\frac{q}{2-q}\,\tanh(\varepsilon) \in (0,1]$. As $V_J(P_\varepsilon A)\leq V_J(A)$ holds trivially, it remains to show
\begin{equation}\label{ineq:J-quantum-variance-bound_rewrite}
    V_J(P_\varepsilon A)  \leq V_J(A)^{1-\vartheta} \left(k^{-1} \left\|d_J A\right\|_q^2\right)^{\vartheta}.
\end{equation}
Since $P_\varepsilon$ commutes with $d_J$,
\begin{equation*}
    V_J(P_\varepsilon A) = \int_0^\infty 2 \left\|d_J P_{\varepsilon+t} A\right\|_2^2 \dif t = \int_0^\infty 2 \left\|P_\varepsilon (d_J P_{t} A)\right\|_2^2 \dif t.
\end{equation*}
Applying hypercontractivity \eqref{eq:hypercontractivity} and norm log-convexity (Lemma~\ref{lem:holder-interp-Schatten}) yields
\begin{equation*}
    \left\|P_\varepsilon (d_J P_{t} A)\right\|_2 \leq \left\|d_J P_{t} A\right\|_{1+e^{-2\varepsilon}} \leq \left\|d_J P_{t} A\right\|_{2}^{1-\vartheta} \left\|d_J P_{t} A\right\|_{q}^{\vartheta}.
\end{equation*}
By the H\"older inequality in the integration over $t$,
\begin{equation*}
      \int_0^\infty 2 \left(\left\|d_J P_{t} A\right\|_{2}^{1-\vartheta} \left\|d_J P_{t} A\right\|_{q}^{\vartheta}\right)^{2} \dif t \leq \left(\int_0^\infty 2 \left\|d_J P_{t} A\right\|_{2}^{2} \dif t\right)^{1-\vartheta}\left(\int_0^\infty 2  \left\|d_J P_{t} A\right\|_{q}^{2} \dif t\right)^\vartheta.
\end{equation*}
The first term reproduces $V_J(A)$ by \eqref{eq:def-local-variance}, while Lemma~\ref{lem:j-derivative-decay} bounds the second term via
\begin{equation*}
    \int_0^\infty 2  \left\|d_J P_{t} A\right\|_{q}^{2} \dif t \leq  \int_0^\infty 2\, e^{-2kt} \left\|d_J  A\right\|_{q}^{2} \dif t = k^{-1} \left\|d_J  A\right\|_{q}^{2}.
\end{equation*}
Combining these bounds yields \eqref{ineq:J-quantum-variance-bound_rewrite}, which completes the proof.
\end{proof}

Finally, combining the general implication (Lemma~\ref{lem:quantum-high-order-influence-var-ineq}) with the short-time variance decay estimates (Lemma~\ref{lem:quantum-high-order-influence-var-decay}), we complete the proof of Theorem~\ref{thm:influence-var}:
\begin{proof}[Proof of Theorem~\ref{thm:influence-var}]
Let $A\in M_2(\Cbb)^{\otimes n}$, $J\subseteq [n]$ with $|J|=k\ge 1$, $p\in [1,2]$, and $q\in [1,2)$.
Optimizing \eqref{ineq:quantum-high-order-influence-var} with $\Rscr = \max\left\{k,\ \frac{\tanh(\varepsilon)}{\varepsilon} \frac{q}{2(2-q)}\ln^+\left(\frac{k\, V_J(A)}{\Inf_J^q[A]^{2/q}}\right)\right\}$ over $0<\varepsilon \leq \arctanh(\frac{2-q}{q})$, and noting that the supremum of $\frac{\tanh(\varepsilon)}{\varepsilon}$ over $\varepsilon$ gives $1$, we obtain
\begin{align*}
    \left\|A\right\|_{\infty}^{2-p}\, \Inf_J^p[A] &\geq 3^{-k(1-\frac{p}{2})}\, V_J(A) \\
    &\hspace{1cm} \cdot\sup_{0<\varepsilon \leq \arctanh(\frac{2-q}{q})}\,\max\left\{k,\, \frac{\tanh(\varepsilon)}{\varepsilon} \,\frac{q}{2(2-q)}\,\ln^+\left(\frac{k\,V_J(A)}{\Inf_J^q[A]^{2/q}}\right)\right\}\\
    &= 3^{-k(1-\frac{p}{2})}\, V_J(A)\, \max\left\{k,\,\frac{q}{2(2-q)} \ln^{+} \left(\frac{k\,V_J(A)}{\Inf_J^q[A]^{2/q}}\right)\right\}.
\end{align*}
This completes the proof of Theorem~\ref{thm:influence-var}.
\end{proof}


\section{Applications}\label{sec:applications}




In this section we derive quantum analogues of several classical inequalities from our two main results: the energy--variance inequality (Theorem~\ref{thm:energy-variance}) and the high-order influence--variance inequality (Theorem~\ref{thm:influence-var}), illustrating the scope of the unified variance-decay framework. We organize these results into two subsections, corresponding to the two main theorems.

\subsection{Energy--variance inequalities}\label{sec:app-energy-var}



Theorem~\ref{thm:energy-variance} yields corollaries that parallel well-known results on the classical Boolean cube: a quantum Talagrand-type isoperimetric inequality (Corollary~\ref{cor:quantum isoper}), a quantum Eldan--Gross inequality (Corollary~\ref{cor:quantum Eldan-Gross}), and a quantum Cordero-Erausquin--Eskenazis-type energy inequality (Corollary~\ref{cor:quantum-p-q-energy}). The proofs are immediate from Theorem~\ref{thm:energy-variance} combined with the $L^p$-Poincar\'e estimate (Proposition~\ref{prop:quantum-p-Poincare}), via appropriate parameter specializations and a few elementary rearrangements and comparisons.

\subsubsection{Talagrand-type isoperimetric inequality}\label{sec:isoperimetry}




Discrete isoperimetric inequalities on the classical Boolean cube formalize the principle that a non-constant Boolean function cannot simultaneously have a small ``boundary'' and nontrivial variance. In \cite{Talagrand1993Iso}, Talagrand established the following quantitative analogue of the Gaussian isoperimetric inequality: for all $f:\{-1,1\}^n\to \{-1,1\}$,
\begin{equation}\label{eq:Talagrand-isoper}
    \Ecal_1[f] = \left\||\nabla f|\right\|_1 \gtrsim \Var(f) \sqrt{1+\ln\left(\frac{1}{\Var(f)}\right)}.
\end{equation}
This was later strengthened by Bobkov~\cite{Bobkov1997isoperimetric}, and extended to an $L^p$ version valid for all $p\in[1,2]$ by Eldan, Kindler, Lifshitz, and Minzer~\cite[Theorem 3.8]{EKLM2025}. More recently Rouz\'e, Wirth, and Zhang raised the question of recovering such isoperimetric phenomena in quantum settings~\cite[Section 6.4]{Rouze2024quantum}. In response, Jiao, Lin, Luo, and Zhou~\cite[Theorem~1.10]{jiao2024quantum} established a quantum analogue of \eqref{eq:Talagrand-isoper} for projections on the quantum Boolean cube. Combining Theorem~\ref{thm:energy-variance} with Proposition~\ref{prop:quantum-p-Poincare}, we establish the following quantum generalization of the $L^p$ inequality in \cite[Theorem 3.8]{EKLM2025}, valid for general observables in $M_2(\Cbb)^{\otimes n}$, thereby extending \cite[Theorem~1.10]{jiao2024quantum}.
\begin{corollary}[Quantum isoperimetric inequality]\label{cor:quantum isoper}
For every $A\in M_2(\Cbb)^{\otimes n}$ with $\left\|A\right\|_{\infty}\leq 1$,
\begin{equation}\label{ineq:quantum isoper}
    \Ecal_p[A] \geq C^{\text{\rm\ref{cor:quantum isoper}}}(p)\, \Var(A) \left[1+ \frac{1}{2}\ln \left(\frac{1}{\Var(A)}\right)\right]^{\frac{p}{2}}
\end{equation}
holds for all $1\leq p< 2$, with $C^{\text{\rm\ref{cor:quantum isoper}}}(p) \gtrsim 2-p$. In particular, if $A$ is a quantum Boolean function (i.e. $A^\ast=A$ and $A^2=\mathbf 1$), then \eqref{ineq:quantum isoper} holds for all $1\le p\le 2$ with $C^{\text{\rm\ref{cor:quantum isoper}}}(p) \gtrsim 1$.
\end{corollary}
\begin{remark}
Under the assumptions, $\Var(A) \leq \left\|A\right\|_{2}^2 \leq \left\|A\right\|_{\infty}^2\leq 1$. If $\Var(A)=0$, we adopt the convention that the right-hand-side of \eqref{ineq:quantum isoper} equals $0$, so the inequality holds trivially.
\end{remark}

We treat separately the general case and the quantum Boolean case. In the first, we specialize Theorem~\ref{thm:energy-variance} to $q=p$ with $1\leq p<2$, then use the quantum $L^p$-Poincar\'e inequality (Proposition~\ref{prop:quantum-p-Poincare}) to compare the $p$-energy and the $L^p$-deviation. In the second, we specialize Theorem~\ref{thm:energy-variance} to $q=1$, then use Proposition~\ref{prop:Boolean-L1-L2-var} to obtain a cleaner expression for the variance.


\begin{proof}
Let $A\in M_2(\Cbb)^{\otimes n}$ with $\left\|A\right\|_{\infty}\leq 1$ and assume w.l.o.g. that $\Var(A)>0$. We first prove the general case. For all $1\leq p< 2$, applying Theorem~\ref{thm:energy-variance} with $q=p$ yields
\begin{equation*}
    \Ecal_p[A] \geq  \Var(A)\,\max\left\{C^{\text{\rm\ref{thm:energy-variance}}}_{1}(p),\, C^{\text{\rm\ref{thm:energy-variance}}}_{2}(p) \left[\frac{p}{2(2-p)} \ln\left(\frac{\Var(A)}{\|A-\tau(A)\|_{p}^2}\right)\right]^{\frac{p}{2}} \right\}.
\end{equation*}
Fix $\lambda\geq 0$. Note that $\lambda+\ln (x)\leq e^{\lambda-1}x$ for $x>0$, and $\Ecal_p[A]=\left\||\nabla A|\right\|_{p}^p$. Hence
\begin{align*}
    &C^{\text{\rm\ref{thm:energy-variance}}}_{1}(p)^{-\frac{2}{p}}\left(\frac{\Ecal_p[A]}{\Var(A)}\right)^{\frac{2}{p}} + C^{\text{\rm\ref{thm:energy-variance}}}_{2}(p)^{-\frac{2}{p}}\left(\frac{\Ecal_p[A]}{\Var(A)}\right)^{\frac{2}{p}} + \frac{p}{2(2-p)} \cdot e^{\lambda-1}\left(\frac{\Ecal_p[A]}{\Var(A)}\right)^{\frac{2}{p}} \\
    &\geq 1+  \frac{p}{2(2-p)} \ln \left(\frac{\Var(A)}{\|A-\tau(A)\|_{p}^2}\right) + \frac{p}{2(2-p)} \left[\lambda +  \ln \left(\left(\frac{\Ecal_p[A]}{\Var(A)}\right)^{\frac{2}{p}}\right) \right]\\
    &= 1+ \frac{1}{2} \ln \left(\frac{1}{\Var(A)}\right) + \frac{p}{2(2-p)}\left[ \lambda + 2\ln\left( \frac{\left\||\nabla A|\right\|_{p}}{\|A-\tau(A)\|_{p}}\right) \right].
\end{align*}
Further by the quantum $L^p$-Poincar\'{e} inequality (Proposition~\ref{prop:quantum-p-Poincare}),
\begin{equation*}
    \left\||\nabla A|\right\|_{p} \geq \frac{3}{2\pi} \|A-\tau(A)\|_{p} \implies \ln\left( \frac{\left\||\nabla A|\right\|_{p}}{\|A-\tau(A)\|_{p}}\right) \geq -\ln\left(\frac{2\pi}{3}\right),
\end{equation*}
so choosing $\lambda=2\ln(2\pi/3)>0$ gives
\begin{equation*}
    \left(C^{\text{\rm\ref{thm:energy-variance}}}_{1}(p)^{-\frac{2}{p}}+C^{\text{\rm\ref{thm:energy-variance}}}_{2}(p)^{-\frac{2}{p}} + e^{2\ln(\frac{2\pi}{3})-1} \frac{p}{2(2-p)} \right)\left(\frac{\Ecal_p[A]}{\Var(A)}\right)^{\frac{2}{p}} \geq 1+ \frac{1}{2} \ln \left(\frac{1}{\Var(A)}\right).
\end{equation*}
Rounding up, we obtain
\begin{equation*}
    \Ecal_p[A] \geq C^{\text{\rm\ref{cor:quantum isoper}}}_1(p)\, \Var(A) \left[1+ \frac{1}{2}\ln \left(\frac{1}{\Var(A)}\right)\right]^{\frac{p}{2}}
\end{equation*}
with $C^{\text{\rm\ref{cor:quantum isoper}}}_1(p) = \left(C^{\text{\rm\ref{thm:energy-variance}}}_{1}(p)^{-2/p}+C^{\text{\rm\ref{thm:energy-variance}}}_{2}(p)^{-2/p} + e^{2\ln(2\pi/3)-1} \frac{p}{2(2-p)} \right)^{-p/2} \gtrsim 2-p$ as desired.

We now turn to the quantum Boolean case. For all $1\leq p\leq 2$, applying Theorem~\ref{thm:energy-variance} with $q=1$ yields
\begin{equation*}
    \Ecal_p[A] \geq  \Var(A)\,\max\left\{C^{\text{\rm\ref{thm:energy-variance}}}_{1}(p),\, C^{\text{\rm\ref{thm:energy-variance}}}_{2}(p) \left[\frac{1}{2} \ln\left(\frac{\Var(A)}{\|A-\tau(A)\|_{1}^2}\right)\right]^{\frac{p}{2}} \right\}.
\end{equation*}
Note that when $A$ is unitary Hermitian, Proposition~\ref{prop:Boolean-L1-L2-var} gives $\Var(A) = \|A-\tau(A)\|_{1}$. Hence
\begin{equation*}
    \left(C^{\text{\rm\ref{thm:energy-variance}}}_{1}(p)^{-\frac{2}{p}}+ C^{\text{\rm\ref{thm:energy-variance}}}_{2}(p)^{-\frac{2}{p}}\right)\left(\frac{\Ecal_p[A]}{\Var(A)}\right)^{\frac{2}{p}} \geq 1+  \frac{1}{2} \ln\left(\frac{\Var(A)}{\|A-\tau(A)\|_{1}^2}\right) = 1+ \frac{1}{2} \ln \left(\frac{1}{\Var(A)}\right).
\end{equation*}
Rounding up, we obtain
\begin{equation*}
    \Ecal_p[A] \geq C^{\text{\rm\ref{cor:quantum isoper}}}_2(p)\, \Var(A) \left[1+\frac{1}{2} \ln \left(\frac{1}{\Var(A)}\right)\right]^{\frac{p}{2}}
\end{equation*}
with $C^{\text{\rm\ref{cor:quantum isoper}}}_2(p)= \left(C^{\text{\rm\ref{thm:energy-variance}}}_{1}(p)^{-2/p}+C^{\text{\rm\ref{thm:energy-variance}}}_{2}(p)^{-2/p} \right)^{-p/2}\gtrsim 1$ as desired.
\end{proof}

\subsubsection{Eldan--Gross inequality}\label{sec:eldan-gross}

Eldan and collaborators recently developed a pathwise stochastic analysis approach, leading to inequalities on the Boolean cube that go beyond the classical hypercontractive framework; see, e.g.,~\cite{Eldan2022,EMR2023,ES2022,EWW2023}. Using this technique, Eldan and Gross~\cite{Eldan2022Concentration} established the following inequality, validating a conjecture of Talagrand~\cite{Talagrand1997OnBA} and strengthening the well-known KKL inequality \cite{KKL1988}: for all $f:\{-1,1\}^n\to \{-1,1\}$ and $1\leq p\leq 2$,
\begin{equation}\label{eq:Eldan--Gross}
    \Ecal_p[f] = \left\||\nabla f|\right\|_p^p \gtrsim \Var(f) \left[\ln\left(2+\frac{e}{\sum_{i}\Inf_i[f]^2}\right)\right]^{\frac{p}{2}}.
\end{equation}
Alternative proofs via combinatorial and semigroup approaches were later given in~\cite{rosenthal2020ramon,EKLM2025,IZ2026}, providing simpler arguments and sharper insights for the inequality. In \cite[Theorem 1.11]{jiao2024quantum} and \cite[Theorem 1.9]{JLZ2025}, Jiao, Lin, Luo and Zhou established quantum analogues of the Eldan--Gross inequality \eqref{eq:Eldan--Gross} for projections. In contrast, by specializing Theorem~\ref{thm:energy-variance} to $q=1$ and combining it with the quantum Talagrand isoperimetric inequality (Corollary~\ref{cor:quantum isoper}), we obtain the following quantum Eldan--Gross inequality \eqref{eq:Eldan--Gross}, valid for general observables in $M_2(\Cbb)^{\otimes n}$.

\begin{corollary}[Quantum Eldan--Gross inequality]\label{cor:quantum Eldan-Gross}
For every $A\in M_2(\Cbb)^{\otimes n}$ with $\left\|A\right\|_{\infty}\leq 1$,
\begin{equation}\label{ineq:quantum Eldan-Gross}
    \Ecal_p[A] \geq C^{\text{\rm\ref{cor:quantum Eldan-Gross}}}(p)\, \Var(A) \left[1+\frac{1}{2}\ln^{+} \left(\frac{1}{\sum_j\Inf_j^{1}[A]^2}\right)\right]^{\frac{p}{2}}
\end{equation}
holds for all $1\leq p< 2$, with $C^{\text{\rm\ref{cor:quantum Eldan-Gross}}}(p)\gtrsim 2-p$. In particular, if $A$ is a quantum Boolean function (i.e. $A^\ast=A$ and $A^2=\mathbf 1$), then \eqref{ineq:quantum Eldan-Gross} holds for all $1\leq p\leq  2$ with $C^{\text{\rm\ref{cor:quantum Eldan-Gross}}}(p)\gtrsim 1$.
\end{corollary}
\begin{proof}
Let $A\in M_2(\Cbb)^{\otimes n}$ with $\left\|A\right\|_{\infty}\leq 1$ and assume w.l.o.g. that $\Var(A)> 0$. For $1\leq p\leq 2$, applying Theorem~\ref{thm:energy-variance} with $q=1$ yields,
\begin{equation*}
   \Ecal_p[A] \geq C^{\text{\rm\ref{thm:energy-variance}}}_{2}(p) \, \Var(A)  \left[\frac{1}{2}\ln^{+}\left(\frac{\Var(A)}{\sum_{j}\Inf_j^{1}[A]^2}\right)\right]^{\frac{p}{2}}.
\end{equation*}
On the other hand, Corollary~\ref{cor:quantum isoper} yields
\begin{equation*}
    \Ecal_p[A] \ge C^{\text{\rm\ref{cor:quantum isoper}}}(p)\, \Var(A)\left[1+\frac{1}{2}\ln\left(\frac{1}{\Var(A)}\right)\right]^{\frac{p}{2}}.
\end{equation*}
Combining these bounds, we deduce
\begin{equation*}
    \left(C^{\text{\rm\ref{cor:quantum isoper}}}(p)^{-\frac{2}{p}}+C^{\text{\rm\ref{thm:energy-variance}}}_{2}(p)^{-\frac{2}{p}}\right)\left(\frac{\Ecal_p[A]}{\Var(A)}\right)^{\frac{2}{p}}  \geq 1+\frac{1}{2}\ln\left(\frac{1}{\Var(A)}\right) + \frac{1}{2}\ln^{+}\left(\frac{\Var(A)}{\sum_{j}\Inf_j^{1}[A]^2}\right).
\end{equation*}
Since $\Var(A)\le \|A\|_2^2\le \|A\|_\infty^2\le 1$,
\begin{align*}
    \ln\left(\frac{1}{\Var(A)}\right) + \ln^{+}\left(\frac{\Var(A)}{\sum_{j}\Inf_j^{1}[A]^2}\right) = \ln\left(\frac{1}{\Var(A)}\right) + \max\left\{\ln\left(\frac{\Var(A)}{\sum_{j}\Inf_j^{1}[A]^2}\right),\ 0 \right\}\\
    \geq \max\left\{\ln\left(\frac{1}{\sum_{j}\Inf_j^{1}[A]^2}\right),\ 0\right\} = \ln^{+}\left(\frac{1}{\sum_{j}\Inf_j^{1}[A]^2}\right).
\end{align*}
Rounding up, we obtain
\begin{equation*}
    \Ecal_p[A] \geq C^{\text{\rm\ref{cor:quantum Eldan-Gross}}}(p) \Var(A) \left[1+\frac{1}{2}\ln^{+} \left(\frac{1}{\sum_j\Inf_j^{1}[A]^2}\right)\right]^{\frac{p}{2}}
\end{equation*}
with $C^{\text{\rm\ref{cor:quantum Eldan-Gross}}}(p) = \left(C^{\text{\rm\ref{cor:quantum isoper}}}(p)^{-2/p}+C^{\text{\rm\ref{thm:energy-variance}}}_{2}(p)^{-2/p} \right)^{-p/2}$, which enjoys the same quantitative estimates as $C^{\text{\rm\ref{cor:quantum isoper}}}(p)$. This completes the proof.
\end{proof}

\subsubsection{Cordero-Erausquin--Eskenazis-type energy inequality}\label{sec:cee}




Motivated by the Talagrand--KKL inequality~\cite{Talagrand1994OnRA}, Cordero-Erausquin and Eskenazis established the following $L^p$--$L^1$ inequality~\cite[Theorem 4]{CE2023}: for all $f:\{-1,1\}^n\to\Cbb$ and $p>1$,
\begin{equation}\label{eq:CEE_scalar_L1_Lp}
     \frac{\||\nabla f|\|_{p}}{1+\sqrt{\ln\bigl(\||\nabla f|\|_{p}/\||\nabla f|\|_{1}\bigr)}} \gtrsim_{p} \|f-\mathbb E_{\mu}f\|_{p}.
\end{equation}
Equivalently, this can be expressed as the $L^p$--$L^1$ energy inequality
\begin{equation}\label{eq:CEE_scalar_L1_Lp_power}
    \frac{\Ecal_p[f]} {\left[1+\ln\bigl(\Ecal_p[f]\,/\,\Ecal_1[f]^p\bigr)\right]^{p/2}}\gtrsim_{p} \|f-\mathbb E_{\mu}f\|_{p}^{p}.
\end{equation}
A quantum analogue the Cordero-Erausquin--Eskenazis inequality \eqref{eq:CEE_scalar_L1_Lp} was established by Jiao, Lin, Luo and Zhou \cite[Theorem 1.8]{JLZ2025}, with a weaker exponent for the logarithmic factor. In contrast, using Theorem~\ref{thm:energy-variance} and Proposition~\ref{prop:quantum-p-Poincare}, we obtain the following quantum analogue of the $L^p$--$L^1$ energy inequality \eqref{eq:CEE_scalar_L1_Lp_power}, allowing a broader range of $q$ while restricted to $p\leq 2$.
\begin{corollary}[Quantum $L^p$-$L^q$ energy inequality]\label{cor:quantum-p-q-energy}
For every $A\in M_2(\Cbb)^{\otimes n}$ with $\left\|A\right\|_{\infty}\leq 1$,
\begin{equation}\label{ineq:Quantum Cordero-Erausquin--Eskenazis}
    \frac{\Ecal_p[A]}{\left[1+ \frac{q}{2(2-q)}\ln^{+} \bigl(\Ecal_p[A]\,/\,\Ecal_q[A]^{2/q}\bigr)\right]^{p/2}} \geq C^{\text{\rm\ref{cor:quantum-p-q-energy}}}(p,q)\, \Var(A)
\end{equation}
holds for all $1\leq q\leq p\leq 2$, with $C^{\text{\rm\ref{cor:quantum-p-q-energy}}}(p,q)\gtrsim 2-q$.
\end{corollary}

Similar to Corollary~\eqref{cor:quantum isoper}, we use the $L^q$-Poincar\'{e} inequality (Proposition~\ref{prop:quantum-p-Poincare}) to bound the $L^q$-deviation from Theorem~\eqref{thm:energy-variance} by $q$-energy, while slightly adjusting the balance between energy and variance terms.

\begin{proof}
Let $A\in M_2(\Cbb)^{\otimes n}$ with $\left\|A\right\|_{\infty}\leq 1$ and $1\leq q\leq p\leq 2$. Assume w.l.o.g. that $\Var(A)> 0$. Applying Theorem \ref{thm:energy-variance} yields
\begin{equation*}
    \Ecal_p[A] \geq  \Var(A)\,\max\left\{C^{\text{\rm\ref{thm:energy-variance}}}_{1}(p),C^{\text{\rm\ref{thm:energy-variance}}}_{2}(p)\cdot \left[ \frac{q}{2(2-q)} \ln \left(\frac{\Var(A)}{\|A-\tau(A)\|_{q}^2}\right)\right]^{\frac{p}{2}}\right\}.
\end{equation*}
Fix $\lambda\geq 0$. Since $\lambda+\ln^+ (x)\leq e^{\lambda-1}x$ for $x>0$, we obtain
\begin{align*}
    &C^{\text{\rm\ref{thm:energy-variance}}}_{1}(p)^{-\frac{2}{p}}\left(\frac{\Ecal_p[A]}{\Var(A)}\right)^{\frac{2}{p}} + C^{\text{\rm\ref{thm:energy-variance}}}_{2}(p)^{-\frac{2}{p}}\left(\frac{\Ecal_p[A]}{\Var(A)}\right)^{\frac{2}{p}} + \frac{q}{2(2-q)}\cdot \frac{p}{2} \cdot e^{\lambda-1}\left(\frac{\Ecal_p[A]}{\Var(A)}\right)^{\frac{2}{p}} \\
    &\geq 1+  \frac{q}{2(2-q)} \ln \left(\frac{\Var(A)}{\|A-\tau(A)\|_{q}^2}\right) + \frac{q}{2(2-q)}\cdot \frac{p}{2} \left[\lambda +  \ln^+ \left(\left(\frac{\Ecal_p[A]}{\Var(A)}\right)^{\frac{2}{p}}\right) \right].
\end{align*}
Recalling $\Ecal_q[A]=\left\||\nabla A|\right\|_{q}^q$ and $\|A-\tau(A)\|_{q}^2 \leq \Var(A)$, it follows that
\begin{align*}
    &\ln \left(\frac{\Var(A)}{\|A-\tau(A)\|_{q}^2}\right) + \frac{p}{2} \left[\lambda +  \ln^+ \left(\left(\frac{\Ecal_p[A]}{\Var(A)}\right)^{\frac{2}{p}}\right) \right]\\
    &= \frac{p}{2} \cdot\lambda + \ln \left(\frac{\Var(A)}{\|A-\tau(A)\|_{q}^2}\right) + \max\left\{\ln \left(\frac{\Ecal_p[A]}{\Var(A)}\right),\ 0\right\}\\
    &=  \max\left\{\ln \left(\frac{\Ecal_p[A]}{\Ecal_q[A]^{2/q}}\right) + \left[ \frac{p}{2}\cdot\lambda + 2\ln\left( \frac{\left\||\nabla A|\right\|_{q}}{\|A-\tau(A)\|_{q}}\right) \right],\ 0\right\}.
\end{align*}
Further by the quantum $L^q$-Poincar\'{e} inequality (Proposition~\ref{prop:quantum-p-Poincare}),
\begin{equation*}
    \left\||\nabla A|\right\|_{q} \geq \frac{3}{2\pi} \|A-\tau(A)\|_{q} \implies \ln\left( \frac{\left\||\nabla A|\right\|_{p}}{\|A-\tau(A)\|_{q}}\right) \geq -\ln\left(\frac{2\pi}{3}\right),
\end{equation*}
so choosing $\lambda=(4/p)\ln(2\pi/3)>0$ gives
\begin{align*}
    &\max\left\{\ln \left(\frac{\Ecal_p[A]}{\Ecal_q[A]^{2/q}}\right) + \left[ \frac{p}{2}\cdot\lambda + 2\ln\left( \frac{\left\||\nabla A|\right\|_{q}}{\|A-\tau(A)\|_{q}}\right) \right],\ 0\right\}\\
    &\geq \max\left\{\ln \left(\frac{\Ecal_p[A]}{\Ecal_q[A]^{2/q}}\right),\ 0\right\} = \ln^{+} \left(\frac{\Ecal_p[A]}{\Ecal_q[A]^{2/q}}\right).
\end{align*}
Combining theses bounds yields
\begin{equation*}
    \left(C^{\text{\rm\ref{thm:energy-variance}}}_{1}(p)^{-\frac{2}{p}}+C^{\text{\rm\ref{thm:energy-variance}}}_{2}(p)^{-\frac{2}{p}} + e^{\frac{4}{p}\ln(\frac{2\pi}{3})-1} \frac{p\,q}{4(2-q)} \right)\left(\frac{\Ecal_p[A]}{\Var(A)}\right)^{\frac{2}{p}} \geq 1+  \frac{q}{2(2-q)} \ln^{+} \left(\frac{\Ecal_p[A]}{\Ecal_q[A]^{2/q}}\right).
\end{equation*}
Rounding up, we obtain
\begin{equation*}
    \frac{\Ecal_p[A]}{\left[1+ \frac{q}{2(2-q)}\ln^{+} \bigl(\Ecal_p[A]\,/\,\Ecal_q[A]^{2/q}\bigr)\right]^{p/2}} \geq C^{\text{\rm\ref{cor:quantum-p-q-energy}}}(p,q)\, \Var(A),
\end{equation*}
with $C^{\text{\rm\ref{cor:quantum-p-q-energy}}}(p,q)=\left(C^{\text{\rm\ref{thm:energy-variance}}}_{1}(p)^{-2/p}+C^{\text{\rm\ref{thm:energy-variance}}}_{2}(p)^{-2/p} + e^{(4/p)\ln(2\pi/3)-1} \frac{p\,q}{4(2-q)} \right)^{-p/2} \gtrsim 2-q$ as desired.
\end{proof}

\subsection{High-order influence--variance inequalities}\label{sec:app-influence-var}


Theorem~\ref{thm:influence-var} provides high-order influence--variance bounds: a quantum Talagrand--KKL-type high-order influence inequality (Corollary~\ref{cor:quantum-p-q-influence}) and a quantum high-order partial isoperimetric inequality (Corollary~\ref{cor:quantum-partial-isoper}).
The proofs are straightforward from Theorem~\ref{thm:influence-var} through suitable parameter choices and elementary rearrangements.

\subsubsection{Talagrand--KKL-type high-order influence inequality}\label{sec:proof-cor15}

The classical theory of influences on the Boolean cube was revolutionized by the famous KKL inequality~\cite{KKL1988}, showing that every nontrivial Boolean function has a ``influential coordinate'' with comparatively large influence. Talagrand later strengthened this principle by establishing the following $L^2$--$L^1$ influence inequality~\cite{Talagrand1994OnRA}: for all $f:\{-1,1\}^n\to \Rbb$,
\begin{equation}\label{ineq:classic Talagrand L1-L2}
    \sum_{i=1}^n\frac{\Inf_i[f]}{1+\ln (\Inf_i[f]/\Inf^1_i[f]^2)} \gtrsim \Var(f).
\end{equation}
This inequality has since been extended to $L^p$ and high-order variants, and to broader settings including hypercontractive spaces and Banach space-valued functions~\cite{CL2012hypercontr,Tanguy2020,CE2023,P2025}. Quantum analogues have also been established, providing noncommutative counterparts of these classical principles~\cite{Rouze2024quantum,Blecher2024geometricinfluencesquantumboolean,JLZ2025}.
Using Theorem~\ref{thm:influence-var}, we obtain the following quantum $L^p$--$L^q$ high-order influence inequality that unifies and extends these principles.

\begin{corollary}[Quantum $L^p$--$L^q$ high-order influence inequality]\label{cor:quantum-p-q-influence}
For every $A\in M_2(\Cbb)^{\otimes n}$ with $\left\|A\right\|_{\infty}\leq 1$, $k\in[n]$, $1\leq p\leq 2$ and $1\leq q<2$,
\begin{equation}\label{eq:quantum Talagrand}
   \sum_{\substack{J\subseteq[n]\\ |J|=k}} \frac{\Inf_J^p[A]}{1+\frac{q}{2(2-q)} \ln^+\bigl(\Inf_J^p[A]/\Inf_J^q[A]^{2/q}\bigr)} \geq C^{\text{\rm\ref{cor:quantum-p-q-influence}}}(k,p,q)\, W^{\geq k}[A]
\end{equation}
holds with $C^{\text{\rm\ref{cor:quantum-p-q-influence}}}(k,p,q)=\left((1+k)\,3^{k(1-\frac{p}{2})}+ \frac{q}{2e(2-q)}\right)^{-1}$.
\end{corollary}
\begin{remark}
When $k=1$, inequality \eqref{eq:quantum Talagrand} reduces to
\begin{equation*}
    \sum_{j\in [n]}
    \frac{\Inf_j^p[A]}{1+\frac{q}{2(2-q)} \ln^+\!\bigl(\Inf_j^p[A]/\Inf_j^q[A]^{2/q}\bigr)}
    \;\geq\;
    C^{\text{\rm\ref{cor:quantum-p-q-influence}}}(1,p,q)\,\Var(A),
\end{equation*}
which recovers and unifies several first-order Talagrand--KKL-type results:
\begin{itemize}
    \item For $q=1$ and $1\leq p\leq 2$, we obtain
    \begin{equation*}
        \sum_{j\in [n]} \frac{\Inf_j^p[A]}{1+\frac{1}{2}\ln^+(\Inf_j^p[A]/\Inf_j^1[A]^2)} \geq C^{\text{\rm\ref{cor:quantum-p-q-influence}}}(1,p,1)\,\Var(A),
    \end{equation*}
    with $C^{\text{\rm\ref{cor:quantum-p-q-influence}}}(1,p,1)\gtrsim 1$. This provides a quantum analogue of the classical $L^p$--$L^1$ inequality of Cordero-Erausquin and Ledoux~\cite[Theorem 6]{CL2012hypercontr}, yielding a quantum extension of \eqref{ineq:classic Talagrand L1-L2}.
    \item For $q=p$ with $1\leq p< 2$, we obtain
    \begin{equation*}
        \sum_{j\in [n]} \frac{\Inf_j^p[A]}{1+\frac{1}{2}\ln^+\bigl(1/\Inf_j^p[A]\bigr)} \geq  C^{\text{\rm\ref{cor:quantum-p-q-influence}}}(1,p,p)\,\Var(A).
    \end{equation*}
    with $C^{\text{\rm\ref{cor:quantum-p-q-influence}}}(1,p,p)\gtrsim 2-p$. This recovers the quantum $L^p$ influence inequality of Blecher, Gao and Xu~\cite[(3.22)]{Blecher2024geometricinfluencesquantumboolean}, improving upon the earlier result of Rouz\'e, Wirth and Zhang~\cite{Rouze2024quantum}.
\end{itemize}
For higher orders $k\geq 1$, inequality \eqref{eq:quantum Talagrand} provides flexible choices of $p$ and $q$, but is quantitatively weaker in the logarithmic exponent compared with \cite[Theorem~1.2]{Blecher2024geometricinfluencesquantumboolean}.
\end{remark}

\begin{proof}
Let $A\in M_2(\Cbb)^{\otimes n}$ with $\|A\|_{\infty}\leq 1$, $k\in[n]$, $1\leq p\leq 2$ and $1\leq q<2$. For $J\subseteq[n]$ with $|J|=k$, we first show that
\begin{equation}\label{eq:J-quantum-Talagrand}
    \frac{\Inf_J^p[A]}{1+\frac{q}{2(2-q)}\ln^+(\Inf_J^p[A]/\Inf_J^q[A]^{2/q})}\geq C^{\text{\rm\ref{cor:quantum-p-q-influence}}}(k,p,q)\, k\, V_J(A).
\end{equation}
Assume w.l.o.g. that $V_J(A)>0$. Applying Theorem~\ref{thm:influence-var} yields
\begin{equation*}
\begin{split}
    \frac{\Inf_J^p[A]}{k\, V_J(A)} &\geq  3^{-k(1-\frac{p}{2})} \max\left\{1,\,\frac{q}{2k(2-q)} \ln^{+}\left( \frac{k\, V_J(A)}{\Inf_J^q[A]^{2/q}}\right)\right\} \\
    &\geq \left(1+k\right)^{-1}3^{-k(1-\frac{p}{2})} \left[1+ \frac{q}{2(2-q)} \ln^{+}\left( \frac{k\, V_J(A)}{\Inf_J^q[A]^{2/q}}\right)\right].
\end{split}
\end{equation*}
Using $\ln^+(x)+\ln^+(y)\geq \ln^+(xy)$ for $x,y>0$ and $\ln^+ (x)\leq e^{-1}x$ for $x>0$, we deduce
\begin{align*}
    &\left(1+k\right)3^{k(1-\frac{p}{2})} \frac{\Inf_J^p[A]}{k\,V_J(A)}+ \frac{q}{2(2-q)}\cdot e^{-1}\,\frac{\Inf_J^p[A]}{k\,V_J(A)}\\
    &\geq   1+ \frac{q}{2(2-q)}\ln^{+}\left( \frac{k\, V_J(A)}{\Inf_J^q[A]^{2/q}}\right) + \frac{q}{2(2-q)}\ln^+ \left(\frac{\Inf_J^p[A]}{k V_J(A)}\right) \\
    &\geq 1 + \frac{q}{2(2-q)} \ln^+ \left(\frac{\Inf_J^p[A]}{\Inf_J^q[A]^{2/q}}\right).
\end{align*}
Rounding up, we obtain
\begin{equation*}
    \frac{\Inf_J^p[A]}{1+\frac{q}{2(2-q)}\ln^+(\Inf_J^p[A]/\Inf_J^q[A]^{2/q})}\geq C^{\text{\rm\ref{cor:quantum-p-q-influence}}}(k,p,q) k\, V_J(A),
\end{equation*}
with $C^{\text{\rm\ref{cor:quantum-p-q-influence}}}_k(p,q):=\left(\left(1+k\right)3^{k(1-\frac{p}{2})}+ \frac{q}{2e(2-q)}\right)^{-1}$ as desired.

Finally, summing \eqref{eq:J-quantum-Talagrand} over all $J\subseteq[n]$ with $|J|=k$, and using Proposition~\ref{prop:sumVarJ-Fourier}, we obtain
\begin{equation*}
    \sum_{\substack{J\subseteq[n]\\ |J|=k}} \frac{\Inf_J^p[A]}{1+\frac{q}{2(2-q)} \ln^+\bigl(\Inf_J^p[A]/\Inf_J^q[A]^{2/q}\bigr)} \geq C^{\text{\rm\ref{cor:quantum-p-q-influence}}}(k,p,q)\, \sum_{\substack{J\subseteq[n]\\ |J|=k}} k\, V_J(A) \geq  C^{\text{\rm\ref{cor:quantum-p-q-influence}}}(k,p,q)\, W^{\geq k}[A],
\end{equation*}
which completes the proof.
\end{proof}

\subsubsection{High-order partial isoperimetric inequality}\label{sec:proof-cor18}

Isoperimetric phenomena on the classical Boolean cube are closely tied to influence inequalities. In~\cite[Section~6.4]{Rouze2024quantum}, Rouz\'e, Wirth and Zhang raised the question of recovering such principle in the quantum setting. Using their Talagrand--KKL-type influence inequality, they deduced an isoperimetric-type bound~\cite[Theorem~6.12]{Rouze2024quantum}, which serves as a quantum analogue of the classical result of Cordero-Erausquin and Ledoux~\cite[Corollary~7]{CL2012hypercontr}: for all projections $A\in M_2(\Cbb)^{\otimes n}$,
\begin{equation}\label{eq:RWZ-quantum-partial-isoper}
    \max_{j\in [n]}\Inf_j^{1}[A]\gtrsim \frac{\Var(A)}{n} \left[\ln\left(\frac{n}{\Var(A)}\right)\right]^{1/2}.
\end{equation}
The authors further conjectured whether an $L^2$ variant of~\eqref{eq:RWZ-quantum-partial-isoper} might hold. In contrast, using Theorem~\ref{thm:influence-var}, we obtain the following quantum high-order partial isoperimetric inequality, which yields sub-$L^2$ high-order refinements of~\eqref{eq:RWZ-quantum-partial-isoper}.

\begin{corollary}[Quantum high-order partial isoperimetric inequality]\label{cor:quantum-partial-isoper}
For every $A\in M_2(\Cbb)^{\otimes n}$ with $\|A\|_{\infty}\leq 1$, $J\subseteq[n]$ with $|J|=k\ge 1$, and $1 \leq p<2$,
\begin{equation}\label{eq:quantum-partial-isoper}
    \Inf_J^p[A] \ge C^{\text{\rm\ref{cor:quantum-partial-isoper}}}(k,p)\left(k\,V_J(A)\right)\left[ 1+\frac12\ln\left(\frac{1}{kV_J(A)}\right)\right],
\end{equation}
and consequently,
\begin{equation}\label{eq:Weight from maxinf}
    \max_{\substack{J\subseteq[n]\\ |J|=k}}\Inf_J^{p}[A]\geq C^{\text{\rm\ref{cor:quantum-partial-isoper}}}(k,p)\,\frac{W^{\ge k}[A]}{\binom{n}{k}} \left[1+\frac{1}{2}\ln\left(\frac{\binom{n}{k}}{W^{\ge k}[A]}\right)\right],
\end{equation}
where $C^{\text{\rm\ref{cor:quantum-partial-isoper}}}(k,p)=\left((1+k)\, 3^{k(1-\frac{p}{2})}+\frac{1}{e(2-p)}\right)^{-1}$.
\end{corollary}
\begin{remark}
When $k=1$, inequality \eqref{eq:Weight from maxinf} reduces to
\begin{equation*}
    \max_{j\in [n]}\Inf_j^{p}[A]\geq C^{\text{\rm\ref{cor:quantum-partial-isoper}}}(1,p)\,\frac{\Var(A)}{n} \left[1+\frac{1}{2}\ln\left(\frac{n}{W^{\ge k}[A]}\right)\right],
\end{equation*}
with $C^{\text{\rm\ref{cor:quantum-partial-isoper}}}(1,p)\gtrsim 2-p$. This strengthens and extends \eqref{eq:RWZ-quantum-partial-isoper} to $L^p$ variants for $p\in [1,2)$, valid for general observables $A\in M_2(\Cbb)^{\otimes n}$. For higher orders $k\geq 1$, our inequality gives a quantitatively weaker KKL-type bound compared to the result of Blecher, Gao and Xu \cite[Theorem 1.2]{Blecher2024geometricinfluencesquantumboolean}, but situates these bounds within a new perspective on partial isoperimetric inequalities.
\end{remark}

%
%

\begin{proof}
Let $A\in M_2(\Cbb)^{\otimes n}$ with $\|A\|_{\infty}\leq 1$ and $1\leq p\leq 2$. For $J\subseteq[n]$ with $|J|=k\geq 1$, we first establish \eqref{eq:quantum-partial-isoper}. Assume w.l.o.g. $V_J(A)>0$. Applying Theorem~\ref{thm:influence-var} with $q=p$ yields
\begin{align*}
    \frac{\Inf_J^p[A]}{k\, V_J(A)} &\geq  3^{-k(1-\frac{p}{2})} \max\left\{1,\,\frac{p}{2k(2-p)} \ln^{+}\left( \frac{k\, V_J(A)}{\Inf_J^p[A]^{2/p}}\right)\right\} \\
    &\geq \left(1+k\right)^{-1}3^{-k(1-\frac{p}{2})} \left[1+ \frac{p}{2(2-p)} \ln^{+}\left( \frac{k\, V_J(A)}{\Inf_J^p[A]^{2/p}}\right)\right].
\end{align*}
Since $\ln^+ x\leq e^{-1}x$ and $k\,V_J(A)\leq \|A\|_2^2\leq \|A\|_\infty^2\leq 1$, we deduce
\begin{equation*}
    \begin{split}
   &\left(1+k\right)\cdot 3^{k(1-\frac{p}{2})} \,\frac{\Inf_J^p[A]}{k\, V_J(A)}+\frac{p}{2(2-p)}\cdot\frac{2}{p}\cdot e^{-1}\,\frac{\Inf_J^p[A]}{k\, V_J(A)}\\
   &\ge 1+ \frac{p}{2(2-p)} \ln^{+}\left( \frac{k\, V_J(A)}{\Inf_J^p[A]^{2/p}}\right)+\frac{p}{2(2-p)}\cdot\frac{2}{p}\,\ln^{+} \left(\frac{\Inf_J^p[A]}{k\, V_J(A)}\right)\\
   &\geq 1+\frac{p}{2(2-p)}\ln \left[\left(\frac{1}{k\, V_J(A)}\right)^{\frac{2}{p}-1}\right]=1+ \frac{1}{2}\ln \left( \frac{1}{kV_J(A)}\right).
    \end{split}
\end{equation*}
Rounding up, we obtain
\begin{equation*}
    \Inf_J^p[A] \ge C^{\text{\rm\ref{cor:quantum-partial-isoper}}}_k(p)\left(k\,V_J(A)\right)\left[ 1+\frac12\ln\left(\frac{1}{k\,V_J(A)}\right)\right],
\end{equation*}
with $C^{\text{\rm\ref{cor:quantum-partial-isoper}}}(k,p)=\left((1+k)\, 3^{k(1-\frac{p}{2})}+\frac{1}{e(2-p)}\right)^{-1}$ as desired.

We now turn to \eqref{eq:Weight from maxinf}. Fix $k\in [n]$, and let $J_\star=\argmax_{|J|=k}V_J(A)$. Proposition~\ref{prop:sumVarJ-Fourier} yields
\begin{equation*}
    k\,V_{J_\star}(A)=\max_{|J|=k}k\,V_J(A)\ge \frac{1}{\binom{n}{k}}\sum_{|J|=k}k\,V_J(A) \ge \frac{W^{\ge k}[A]}{\binom{n}{k}}.
\end{equation*}
Since $x\mapsto x\left(1-\frac{1}{2}\ln x\right)$ is increasing on $(0,1]$,
for $0< W^{\ge k}[A]/\binom{n}{k}\leq k\,V_{J_\star}(A)\leq 1$ we obtain
\begin{align*}
    \max_{|J|=k}\Inf_J^p[A] &\ge C^{\text{\rm\ref{cor:quantum-partial-isoper}}}_k(p)\left(k\,V_{J_\star}(A)\right)\left[ 1+\frac12\ln\left(\frac{1}{k\,V_{J_\star}(A)}\right)\right]\\
    &\geq C^{\text{\rm\ref{cor:quantum-partial-isoper}}}(k,p)\, \frac{W^{\ge k}[A]}{\binom{n}{k}}\left[1+\frac{1}{2}\ln\left(\frac{\binom{n}{k}}{W^{\ge k}[A]}\right)\right],
\end{align*}
which completes the proof.
\end{proof}

\section*{Acknowledgments}
\addcontentsline{toc}{section}{Acknowledgments}
Fan Chang is supported by the NSFC under grant 124B2019 and the Institute for Basic Science (IBS-R029-C4). The authors are grateful to Haonan Zhang and Sijie Luo for reading an earlier draft of this paper and for providing valuable comments and suggestions. They also thank the referees for their careful reading and helpful remarks.

\bibliographystyle{abbrv}
\bibliography{reference}
\addcontentsline{toc}{section}{References}
\appendix

\section{Basic Analytic Facts}\label{subsec:prelim-analytic}

We collect several analytic inequalities used repeatedly in the sequel.


\begin{lemma}[Noncommutative vector-valued H\"older inequality]\label{lem:nc-holder}
Let $A_{j},B_{j}\in  M_2(\mathbb{C})^{\otimes n}$, $m\in \Nbb$ and $j\in [m]$. Then for every $p,q\in [1,\infty]$ such that $\frac1p+\frac1q=1$, we have
\begin{equation*}
    \left|\tau\bigg(\sum_{j=1}^{m} A_j^{*} B_j\bigg)\right| \leq \left\|\bigg(\sum_{j=1}^{m} |A_j|^2\bigg)^{1/2}\right\|_{p} \left\|\bigg(\sum_{j=1}^{m} |B_j|^2\bigg)^{1/2}\right\|_{q}
\end{equation*}
\end{lemma}


\begin{lemma}[Log-convexity of the normalized Schatten norm]\label{lem:holder-interp-Schatten}
For $A\in M_2(\mathbb{C})^{\otimes n}$, $r\mapsto \|A\|_{1/r}$ is log-convex. More specifically, if $p_1, p_2 \in [1,\infty]$ and $\frac{1}{p_\theta} = \frac{1-\theta}{p_1} + \frac{\theta}{p_2}$ $\theta\in [0,1]$, then
\begin{equation*}
    \|A\|_{p_\theta} \leq \|A\|_{p_1}^{1-\theta} \cdot \|A\|_{p_2}^{\theta}.
\end{equation*}
In particular,
\begin{itemize}
    \item for $q \in [2,\infty]$, we have $\frac{1}{q} = \frac{2}{q} \cdot \frac{1}{2} + (1-\frac{2}{q}) \frac{1}{\infty}$, and hence
    \begin{equation*}
        \left\|A\right\|_{q} \leq \left\|A\right\|_{2}^{\frac{2}{q}} \left\|A\right\|_{\infty}^{\frac{q-2}{q}}.
    \end{equation*}
    \item for $0<\epsilon \leq \arctanh(\frac{2-q}{q})$, we have $\frac{1}{1+e^{-2\varepsilon}}=\frac{1-\theta}{2}+\frac{\theta}{q}$ with $\theta=\frac{q}{2-q}\,\tanh(\varepsilon)$, and hence
    \begin{equation*}
        \left\|A\right\|_{1+e^{-2\varepsilon}} \leq \left\|A\right\|_{2}^{1-\theta} \left\|A\right\|_{q}^{\theta}.
    \end{equation*}
\end{itemize}
\end{lemma}


\begin{lemma}[Kadison--Schwarz inequality~\cite{kadison1952generalized}]\label{lem:Kadison--Schwarz}
Let $\Phi$ be a unital positive map between $C^*$-algebras. Then for every $A$ in the domain of $\Phi$, we have $\Phi(|A|^2)\ge |\Phi(A)|^2$.
\end{lemma}

\begin{lemma}[Matrix subadditivity inequality~\cite{BU2007subadditivity}]\label{lem:subadditivity}
Let $A, B$ be nonnegative Hermitian matrices, $f:[0,\infty)\to [0,\infty)$ be a concave function, and $\|\cdot\|$ be a unitary invariant trace norm. Then
\begin{equation*}
    \left\|f(A+B)\right\| \leq \left\|f(A)\right\| + \left\|f(B)\right\|.
\end{equation*}
In particular, taking $f(x)=x^{r}$ with $0<r\leq 1$ and $\|\cdot\| = \tau(|\cdot|)$, we have
\begin{equation*}
    \tau(|A+B|^{r}) \leq  \tau(|A|^{r}) + \tau(|B|^{r}).
\end{equation*}
\end{lemma}

\section{Useful Estimates}

\begin{lemma}\label{lem:beta_upper}
Let $a\in[1/2,1]$. Then for all $r\in[0,1]$,
\begin{equation*}
    \Beta\left(\frac{3}{2}\,\frac{r}{1+r};a,a\right)\le \frac{1}{a}\left(\frac32\, r\right)^{a}.
\end{equation*}
\end{lemma}
\begin{proof}
Substituting $t=\frac{3}{2} \frac{u}{1+u}$, we obtain
\begin{equation*}
    \Beta\left(\frac{3}{2}\,\frac{r}{1+r};a,a\right) = \int_0^{\frac{3}{2}\frac{r}{1+r}} t^{a-1}(1-t)^{a-1}\dif t = \left(\frac{3}{2}\right)^a \int_0^r u^{a-1} \left(1+u\right)^{-2a}\left(1-\frac{u}{2}\right)^{a-1} \dif u.
\end{equation*}
Since $u\in [0,1]$, we have
\begin{equation*}
    u-u^2 \geq 0 \iff \left(2-u\right)\left(1+u\right)\geq 2 \iff \left(1-\frac{u}{2}\right)\geq \left(1+u\right)^{-1}.
\end{equation*}
Further note that $a-1\leq 0$ and $1-3a\leq 0$, we have
\begin{equation*}
    \left(1+u\right)^{-2a}\left(1-\frac{u}{2}\right)^{a-1} \leq \left(1+u\right)^{-2a}\left(1+u\right)^{1-a} = \left(1+u\right)^{1-3a} \leq 1.
\end{equation*}
Hence we obtain
\begin{equation*}
    \Beta\left(\frac{3}{2}\,\frac{r}{1+r};a,a\right) \leq \left(\frac{3}{2}\right)^a \int_0^r u^{a-1} \dif u = \frac{1}{a}\left(\frac32\, r\right)^{a}
\end{equation*}
as desired.
\end{proof}


\begin{lemma}\label{lem:phi_max_location}
Let $a\in[1/2,1]$. Then
\begin{equation*}
    \argmax_{x\geq 0} \frac{1-e^{-x}}{x^{a}} \leq \frac{2(1-a)}{a}.
\end{equation*}
\end{lemma}
\begin{proof}
Write $\varphi_a(x):=\frac{1-e^{-x}}{x^{a}}$. Then for $x>0$,
\begin{equation*}
    \varphi_a^{\prime}(x) = x^{-a-1} e^{-x} \left( x- a\left(e^x-1\right)\right).
\end{equation*}
\begin{itemize}
    \item When $a=1$, for all $x>0$, note that $e^x-1\geq x$, we have
    \begin{equation*}
        \varphi_1^{\prime}(x) = x^{-a-1} e^{-x} \left( x- \left(e^x-1\right)\right)  \leq 0.
    \end{equation*}
    Hence $\varphi_1(x)$ is decreasing on $(0,\infty)$, and equivalently, maximized at $x=0 = \frac{2(1-a)}{a}$.
    \item When $a\in [1/2,1)$, $\varphi_a(x)$ is maximized by its unique critical point $x_*>0$, which satisfies
    \begin{equation*}
        \varphi_a^{\prime}(x_*)=0 \iff x_* = a\left(e^{x_*}-1\right).
    \end{equation*}
    Further using $e^{x}-1 \geq x+\frac{1}{2}x^2$ for $x>0$, we have
    \begin{equation*}
        x_* = a\left(e^{x_*}-1\right) \geq a\left(x_*+\frac{1}{2}x_*^2\right) \iff 0\leq x_* \leq \frac{2(1-a)}{a}.
    \end{equation*}
\end{itemize}
Hence, in both cases, the maximizer of $\varphi_a$ in $[0,\infty)$ is bounded by $\frac{2(1-a)}{a}$ as desired.
\end{proof}

\end{document}